\documentclass[a4paper,11 pt]{amsart}
\usepackage{amsfonts}
\usepackage{amssymb}
\usepackage[utf8]{inputenc}
\usepackage{amsmath}
\usepackage{pdflscape}
\usepackage{graphicx}
\usepackage[colorlinks=true, linkcolor=red, citecolor=blue]{hyperref}
\usepackage[]{epsfig}
\usepackage[]{pstricks}
\usepackage{tikz}

\numberwithin{equation}{section}
\newtheorem{theorem}{Theorem}[section]
\newtheorem{proposition}[theorem]{Proposition}

\newtheorem{remark}{Remark}[section]
\newtheorem{claim}{Claim}
\newtheorem{definition}{Definition}
\newtheorem{remarks}{Remarks}

\newcommand{\N}{\mbox{$I \kern -4pt N$}}

\newcommand{\Q}{\mbox{$Q \kern -8pt I$}}
\newcommand{\R}{\mbox{$I \kern -4pt R$}}
\newcommand{\C}{\mbox{$C \kern -8pt I$}}

\def\R{{\bf R}}
\numberwithin{equation}{section}
\makeatletter
\@namedef{subjclassname@2020}{\textup{2020} Mathematics Subject Classification}
\makeatother

\begin{document}
\title[Insensitizing controls for 4NLS]{Controls insensitizing the norm of solution of a Schrödinger type system with mixed dispersion}
\author[Capistrano--Filho]{Roberto de A. Capistrano--Filho*}\thanks{*Corresponding author: roberto.capistranofilho@ufpe.br}
\address{Departamento de Matem\'atica,  Universidade Federal de Pernambuco (UFPE), 50740-545, Recife (PE), Brazil.}
\email{roberto.capistranofilho@ufpe.br}
\author[Tanaka]{Thiago Yukio Tanaka}
\address{Departamento de Matemática, Universidade Federal Rural de Pernambuco (UFRPE), 52171-900, Recife (PE), Brazil.}
\email{thiago.tanaka@ufrpe.br}
\thanks{Capistrano–Filho was partially supported by CAPES/MATH-AMSUD \#8881.520205/2020-01, CAPES/PRINT \#88887.311962/2018-00, CAPES/COFECUB \#88887.879175/2023-00, CNPq, Brazil, \#421573/2023-6, \#307808/2021-1 and \#401003/2022-1 and Propesqi (UFPE)}

\begin{abstract}
The main goal of this manuscript is to prove the existence of insensitizing controls for the fourth-order dispersive nonlinear Schrödinger equation with cubic nonlinearity.  To obtain the main result we prove a null controllability property for a coupled fourth-order Schrödinger cascade type system with zero-order coupling which is equivalent to the insensitizing control problem. Precisely, employing a new Carleman estimate, we first obtain a null controllability result for the linearized system around zero, and then the null controllability for the nonlinear case is extended using an inverse mapping theorem.
\end{abstract}
\maketitle

\section{Introduction}

\subsection{Setting of the problem} Inspired by the difficulties in finding data in distributed system applications, Lions \cite{lions1990quelques} introduced the topic of insensitizing controls. Precisely, this kind of problem deals with the existence of controls making a functional (depending on the solution) insensible to small perturbations of the initial data. Considering some particular functional, it has been proven that this problem is equivalent to control properties of cascade systems \cite{bodart1995controls,guerrero2007null}.

The insensitivity can be defined in two different ways: An approximate problem or an exact problem. Approximate insensitivity is equivalent to the approximate controllability of the cascade system, while exact insensitivity is equivalent to its null controllability. Before giving the reader more details about it and a state-of-the-art related to these problems, let us introduce the model we want to study.   

In this article, we address the exact insensitizing problem for the cubic fourth-order Schrödinger equation with mixed dispersion, the so-called fourth-order nonlinear Schrödinger system (4NLS)
\begin{equation}
\label{fourtha}
iu_t +u_{xx}-u_{xxxx}=\lambda |u|^2u,
\end{equation}
where $x, t \in\mathbb{R}$ and $u(x,t)$ is a complex-valued function. Equation \eqref{fourtha} has been introduced by Karpman \cite{Karpman} and Karpman and Shagalov \cite{KarSha} to take into account the role of small fourth-order dispersion terms in the propagation of intense laser beams in a bulk medium with Kerr nonlinearity. Equation \eqref{fourtha} arises in many scientific fields such as quantum mechanics, nonlinear optics, and plasma physics, and has been intensively studied with fruitful references (see \cite{Ben,Karpman,Paus1} and references therein).

To introduce our problem, consider $\Omega : = (0,L) \subset \mathbb{R}$ be an interval and assume that $T>0$. We will use the following notations $Q_T=\Omega\times(0,T)$, $\Sigma=\partial\Omega \times(0,T)$ and  $\omega_T=\omega\times(0,T)$, where $\omega\subset\Omega$ is the so-called \textit{control domain}. Let us consider  the system
\begin{equation}\label{biharmonic}
\left\{
\begin{array}{lll}
      i u_t  + u_{xx} - u_{xxxx} - \zeta |u|^2 u = f + 1_{\omega}h,  &\text{in}& Q_T,\\
      u(t,0) = u(t,L) = u_x(t,0) = u_x(t,L) = 0, &\text{on}& t \in (0,T), \\
      u(0,x)=u_0(x) +\tau \hat{u}_0(x), &\text{in}&\Omega,
\end{array}
\right.
   \end{equation}
where $\zeta\in\mathbb{R}$, $\tau$ is an unknown (small) real number, $h$ stands for the control, $u$ is the state function and $\hat{u}_0(x)$ is an unknown function. 

Let $J:\mathbb{R}\times L^2(q_T)\rightarrow \mathbb{R}$ be a functional (called {\it sentinel functional}) defined by
\begin{equation}\label{def:jfunc}
J(\tau, h) := \frac{1}{2} \iint_{\mathcal{O}_T}| u(x,t;\tau,h)|^{2}dxdt,
\end{equation}
	where $u = u(x,t;\tau,h)$ is the corresponding solution of \eqref{biharmonic} associated to $\tau$, $h$ is the control function and $\mathcal{O}_T=\mathcal{O}\times(0,T)$, where $\mathcal{O}$ is the so-called \textit{observation domain}.  Thus, our objective can be expressed in the definition below.

\begin{definition}[Insensitizing controls]\label{def:insen}
	Let $u_0\in L^2(\Omega)$ and $f\in L^2(Q_T)$. We say that a control $h$ insensitizes the functional $J$, associated with the solution $u(x,t;\tau,h)$ of \eqref{biharmonic}, if
\begin{equation}\label{cond:insens}
	\frac{\partial}{\partial \tau} J(\tau,h) \bigg|_{\tau = 0}  = 0, 
	\ \ \forall \ \widehat{u}_{0} \in L^{2}(\Omega)\ \ 
	\hbox{with} \ \  \|\widehat{u}_{0}\|_{L^{2}(\Omega)} = 1.
\end{equation}
\end{definition}

The definition \eqref{def:jfunc} above can be seen as this: the sentinel does not detect (small) variations of the initial data $u_0$ provoked by the unknown (small) perturbation $\tau \hat u_0$ in the observation domain $\mathcal{O}$ when the system evolves from a time $t=0$ to a time $t = T$.

It is important to point out that in \cite{lissyprivatsimpore2018} the authors discuss the motivation for insensitizing controls considering linear and semilinear heat equations with partially unknown domains and state that the existence of such controls is important for maintaining the stability of the solutions and robustness against external variations and uncertainties. They suggest that it can ensure the system evolution is unaffected by parameter perturbations, guaranteeing predictable and stable operation, particularly in industrial and environmental contexts where conditions may vary.

Based on this motivation, we can also consider the general application of insensitizing controls to the fourth order Schrödinger equation similarly since the existence of such controls aims to ensure robustness against perturbations, precision in quantum experiments, and stability in open systems. They may ensure that quantum systems maintain their desired properties and operate predictably and stably, even amidst uncertainties and noise. Small variations in potentials or boundary conditions can cause significant changes in system behavior in quantum systems like atoms or molecules. In open systems interacting with external environments, smoothing noises and decoherence effects preserve quantum state coherence. Quantum technologies, such as sensors and atomic clocks, are extremely sensitive to external variations; thus, maintaining the precision and reliability of these devices enhances their performance and applicability across various technological fields.

Since the fourth-order Schrödinger equation incorporates phenomena with dispersion effects and complex interactions, these controls can ensure precision and stability in materials with strong spin-orbit interactions, exotic quantum systems, or advanced dispersion phenomena. It may be applied in open systems to mitigate noise and decoherence effects, ensuring precision in sensitive quantum technologies. Therefore, they are crucial for the robustness and accuracy of complex quantum systems. Now, before presenting the results of our work, let us give some previous results concerning the insensitizing control problems.

\subsection{Insensitizing control problems for PDEs} As we have mentioned, the first time that insensitizing problem was approached was in the early '90s by Lions \cite{lions1990quelques,lions1992sentinelles}, where the author studied second and fourth-order parabolic equations in limited domains considering a functional with the local $L^2-$norm of the solution of a system with null initial data and where the control domain (located internally) intersect the observation domain (the set where we want to analyze the functional). 

Since then, variations of this problem have been considered, i.e., to find controls that turn a functional depending on the solution (or some derivative) insensitive to small perturbations depending on the initial data. We will give a brief state of the art to the reader, precisely, we will present a sample of the insensitizing problems for partial differential equations (PDEs) and some control results to the system \eqref{fourtha}.

The first mathematical results concerned the insensitivity of the $L^2-$norm of the solution restricted to a subdomain, called the observatory. In \cite{teresa2000insensitizing}, the author proves that insensitizing control problems cannot be solved for every initial data. Additionally, de Tereza \cite{de2009identification} used a global Carleman estimate approach to get the existence of exact insensitizing controls for a semilinear heat equation. Still, concerning the semilinear heat equation, Bodart \textit{et al.}  \cite{bodart1995controls,bodart2004existence} proved the existence of insensitizing controls for this system with nonlinear boundary Fourier conditions. 

Concerning the variations of the sentinel functional, in \cite{guerrero2007null}, the author considers a functional involving the gradient of the state for a linear heat system, in the same way, Guerrero \cite{guerrero2007controllability}  treated the case of the sentinel with the curl of the solution of a Stokes system.  About the wave equation, Alabau-Boussouira \cite{Alabau-Boussouira}, showed the exact controllability, by a reduced number of controls, of coupled cascade systems of PDE’s and the existence of exact insensitizing controls for the scalar wave equation. She gave a necessary and sufficient condition for the observability of abstract-coupled cascade hyperbolic systems by a single observation, the observation operator being either bounded or unbounded.

A variation of the (exact) control strategy was presented in \cite{bodart1995controls}, where the authors considered an approximated insensitizing problem (called $\epsilon-$insensitizing control) for a nonhomogeneous heat equation. We observe that by smoothing the control strategy it was possible to prove in \cite{micuortegateresa2004} positive results where the control domain and observation region do not intersect each other. Also in \cite{teresa1997} the author proved insensitizing control results on unbounded domains, in \cite{lissyprivatsimpore2018} the authors treated insensitizing controls for both linear and semilinear heat equation but with a partially unknown domain, finally see \cite{mumingyingao2019} for the semilinear parabolic equation with dynamic boundary conditions. It is important to point out that in \cite{tanaka2019} the second author also treated it with a gradient-type sentinel associated with the solutions of a nonlinear Ginzburg-Landau equation.

Concerning the structure/type of the equations/systems, many variations were considered. In \cite{carreno2017insensitizing} the author treated insensitizing controls for the Boussinesq systems and in \cite{calsavaracarrenocerpa2016} the authors proved a result for a phase field system. In \cite{gao2014} it is considered a Cahn-Hilliard equation of fourth order with superlinear nonlinearity and in \cite{enriquegarciaaxel2003} the authors proved insensitizing (exact and approximated) controls for a large-scale ocean circulation model.

To finalize this small sample of the state of the art, we cite Bodart \textit{et al.} \cite{bodart2004existence,bodart2004insensitizing} that studied systems with nonlinearities with certain superlinear growth and nonlinear terms depending on the state and its gradient. For a dispersive problem, we can cite Kumar and Chong \cite{kumarkil2017} which worked with the KdV-Burgers equation.  Finally, let us mention a recent work due to Lopez-Garcia \textit{et al.} \cite{de_tereza_lopes_mercado}. In this work, the authors presented a control problem for a cascade system of two linear $N-$dimensional Schrödinger equations. They address the problem of null controllability using a control supported in a region not satisfying the classical geometrical control condition. The proof is based on the application of a Carleman estimate with degenerate weights to each one of the equations and a careful analysis of the system to prove null controllability with only one control force.


We caution that the literature is vast and one can see the references cited previously for the existence of the insensitizing controls for other types of PDEs. 


\subsection{Main results}
 The main goal of this paper is to close the gap that was missing when discussing the insensitizing control for the Schrödinger type equation. Here, we present the insensitizing control for the Schrödinger type equation with mixed dispersion. Precisely,  we are interested in proving the existence of a control $h$ which insensitizes the functional $J$ defined by \eqref{def:jfunc}.  The first result of this article can be read as follows.
\begin{theorem}\label{thm:insen} 
	Assume that $\omega \cap \mathcal{O} \neq \emptyset$ and $u_{0} \equiv 0$. There exists a constant $C>0$  and $\delta>0$ such that for any $f$ satisfying
	$$\|e^{C/t}f\|_{L^2(Q_T)}\leq \delta,$$
	one can find a control $h(x,t)=:h\in L^2(q_T)$ which insensitizes the functional $J$ defined by \eqref{def:jfunc}, in the sense of Definition \ref{def:insen}.
\end{theorem}

As mentioned at the beginning of this work, the existence of insensitizing controls for \eqref{biharmonic} can be defined equivalently by means of a null controllability problem for a cascade type system similar to the initial \eqref{biharmonic}. Indeed, this process can be systematized when, in defining the functional in \eqref{def:jfunc}, we study the condition given by \eqref{cond:insens}. Precisely, by calculating the derivative in the sense of Gâteaux for the functional $J$ restricted to $\tau = 0$ and given that the functional is the localized $L^2$ norm in $\mathcal{O}$ of the solution $u$, the insensitizing condition \eqref{cond:insens} implies that we can reformulate this by a null controllability problem for a coupled system which is \eqref{eq:nonlinear_system}. Through these calculations, it can be proved that the left-hand side of the second equation in \eqref{eq:nonlinear_system} is the adjoint state of the derivative of \eqref{biharmonic} with respect to $\tau$ (at $\tau = 0$). Thus, the right-hand side couples this last equation with the localized state $1_{\mathcal{O}}u$. We remark that different definitions for the functional also imply different coupling terms (see \cite{guerrero2007null}, we also commented this in Section \ref{sb1}). After this, to ensure that the insensitivity condition of the functional is satisfied, i.e., to ensure \eqref{cond:insens}, it is sufficient to ensure that $v|_{t = 0} \equiv 0$ in $\Omega$\footnote{See \cite{bodart1995controls,lions1990quelques} for a rigorous deduction of this fact and \cite{zl} for an explicit computation to obtain the cascade system with a general nonlinearity for a Ginzburg-Landau equation.}.

So, in this spirit, due to the choice of $J$, we will reformulate our goal as a partial null controllability problem to the nonlinear system of cascade type associated with \eqref{biharmonic}. In other words:

\vspace{0.2cm}
\noindent\textbf{Problem $\mathcal{A}$:} \textit{Can we find a control $h(x,t)=h\in L^2(q_T)$ such that the solutions $(u,v)$ of the following optimality coupled system 
	\begin{equation}\label{eq:nonlinear_system}
	\left\{\begin{array}{lll}
      i u_t  + u_{xx} - u_{xxxx} - \zeta |u|^2 u = f + 1_{\omega}h,  &\text{in}& Q_T,\\
      i  v_t  + v_{xx} - v_{xxxx} - \overline{\zeta} \overline{u}^2\overline{v} - 2\overline{\zeta} |u|^2 v = 1_{\mathcal{O}}u,   &\text{in}& Q_T,\\
      u(t,0) = u(t,L) = v(t,0) = v(t,L) = 0,&\text{on}& t \in (0,T), \\
      u_x(t,0) = u_x(t,L) = v_x(t,0) = v_x(t,L) = 0,&\text{on}& t \in (0,T), \\
         u(0,x)=u_0(x), \quad v(T,x) = 0, &\text{in}&\Omega,
	\end{array} \right.
	\end{equation}
	satisfies, in the time $t=0$, $v|_{t=0}=0$?}
\vspace{0.2cm}

The answer to such a question motivates the next theorem, which is the main result of this paper.

\begin{theorem}\label{thm:control} Assume that $\omega \cap \mathcal{O} \neq \emptyset$ and the initial data $u_{0} \equiv 0$. Then, there exists positive constants $C$ and $\delta$, depending on $\omega$, $\Omega$, $\mathcal{O}$, $\zeta$  and $T$, such that for any $f$ satisfying $$\|e^{C/t}f\|_{L^2(Q_T)}\leq \delta,$$ there exists a control $h \in L^{2}(q_T)$ such that the corresponding solution $(u,v,h)$ of \eqref{eq:nonlinear_system} satisfies $v|_{t=0} = 0$ in $\Omega$. 
\end{theorem}

Thus, to find controls insensitizing the functional $J$, that is, to prove Theorem \ref{thm:insen}, it is sufficient to prove the partial null controllability result given in Theorem \ref{thm:control}. Therefore, from now on, we will concentrate on proving Theorem \ref{thm:control}.

\subsection{Heuristic and structure of the manuscript} Let us now explain the ideas to prove the results introduced in the last subsection. The main strategy adopted is based on duality arguments (see, e.g. \cite{DolRus1977,lions1}). Roughly speaking,  we prove suitable observability inequalities for the solutions of an adjoint system, where the main tool is a new Carleman estimate. This Carleman estimate with the right-hand side in weight Sobolev spaces will be the key point to deal with the coupling terms of the linear system associated with \eqref{eq:nonlinear_system}.

In detail, to prove Theorem \ref{thm:control} we will first prove a null controllability result for the linearized system associated with \eqref{eq:nonlinear_system} around zero, which is given by
\begin{equation}\label{eq:linearized_systeminversa}
	\left\{\begin{array}{lll}
      i  u_t  +  u_{xx} - u_{xxxx}  = f^0 + 1_{\omega}h,  &\text{in}& Q_T,\\
      i v_t  + v_{xx} - v_{xxxx} = f^1 +1_{\mathcal{O}}u,   &\text{in}& Q_T,\\
      u(t,0) = u(t,L) = v(t,0) = v(t,L) = 0, &\text{on}& t \in (0,T),\\  u_x(t,0) =  u_x(t,L) = v_x(t,0) = v_x(t,L) = 0,&\text{on}& t \in (0,T), \\
      u(0,x)=u_0(x), v(T,x) = 0, &\text{in}&\Omega.
	\end{array} \right.
	\end{equation}
Here, $f^0$ and $f^1$ are (small) source terms in appropriated $L^p$-weighted spaces. In order to prove it, we consider the adjoint system of \eqref{eq:linearized_systeminversa}, namely, 
\begin{equation}\label{adjuntoinversaa}
	\left\{\begin{array}{lll}
      i \varphi_t  +  \varphi_{xx} - \varphi_{xxxx} = 1_{\mathcal{O}}\psi +g^0,  &\text{in}& Q_T,\\
      i \psi_t  +  \psi_{xx}  - \psi_{xxxx} = g^1,   &\text{in}& Q_T,\\
      \varphi(t,0) = \varphi(t,L) =  \varphi_x(t,0) = \varphi_x(t,L) = 0, &\text{on}& t \in (0,T),\\  \psi(t,0) = \psi(t,L) = \psi_x(t,0) = \psi_x(t,L) = 0, &\text{on}& t \in (0,T), \\
      \varphi(T,x)=0, \psi(0,x) = \psi_0(x), &\text{in}&\Omega.
	\end{array} \right.
	\end{equation}
With this in hand, we can prove an \textit{observability inequality}, with aspects like the one below,
\begin{equation}\label{observabilidadeinversa}
    \iint_{Q_T} \rho_1(|\varphi|^2 + |\psi|^2) dxdt \\ \leq C \iint_{\omega_T} \rho_2 |\varphi|^2dxdt + \iint_{Q_T} \rho_3 (|g^0|^2 +|g^1|^2)dxdt,
\end{equation}
where $\rho_i$, $i=1,2,3$, are appropriate weights functions. Then, by duality approach, the desired partial null controllability property is a direct consequence of the \eqref{observabilidadeinversa} and can be read as follows.
\begin{theorem}\label{thm:control1}Assuming that $\omega \cap \mathcal{O} \neq \emptyset$ and the initial data $u_{0} \equiv 0$,  there exists a positive constant $C$, depending on $\delta$, $\omega$, $\Omega$, $\mathcal{O}$ and $T$, such that for  $f^0$ and $f^1$, in suitable weighted spaces, one can find a control $h$ such that the associated solution $(u,v)$ of \eqref{eq:linearized_systeminversa}
satisfies $v|_{t=0}\equiv 0$ in $\Omega$.
\end{theorem}
The last step is to use an inverse mapping theorem to extend the previous result to the nonlinear system. 
\begin{remarks}Finally, the following comments are now given in order:
\begin{itemize}
\item[1.] In the Definition \ref{def:insen}, it is important to point out that the data $u_0$ will be taken conveniently, precisely, will be taken such that the functional $J$, given by \eqref{def:jfunc}, is well-defined.
\item[2.] On the Theorem \ref{thm:control}, the smallness of $f$ is related to the fact that we will apply a local inversion argument, that is, we first study this problem when a linearized form of equation \eqref{eq:nonlinear_system} is considered and then we apply a local inversion mapping theorem.
\item[3.]We claim that Theorem \ref{thm:insen} is equivalent to Theorem \ref{thm:control}. In fact, considering $(u,v)$ solution of \eqref{eq:nonlinear_system} and using the boundary conditions of \eqref{eq:nonlinear_system}, a direct calculation leads us to the following
\begin{equation}\label{expression}
\left.\frac{\partial}{\partial \tau} J(\tau, h))\right|_{\tau=0}=-\operatorname{Re} \int_{\Omega}i\hat{u}_{0} v(0) d x.
\end{equation}
Therefore, we can conclude that the left-hand side of \eqref{expression} is zero, for all $ \widehat{u}_{0} \in L^{2}(\Omega)$ with $\|\widehat{u}_{0}\|_{L^{2}(\Omega)} = 1$, if and only if, $v(0)=0$ in $\Omega$. \color{black}
\item[4.] It is worth mentioning that, in our work, we need a Carleman estimate with internal observation, differently from what was proven by Zheng \cite{zheng}. Another interesting point is that in Zheng's work, he proved the regularity of the solution of the 4NLS in the class $$C^1([0,T]; L^2(\Omega) ) \cap C^0([0,T];H^3(\Omega) \cap H_0^2(\Omega)),$$ which is also different in our case, we need more regular solutions (see Section \ref{Sec1}) to help us to use the inverse mapping theorem.
\item[5.] Finally, observe that our sentinel functional $J$ is defined in the sense of $L^2-$norm. If we want to insensitize a functional with a norm greater than $L^2$, for example, $\partial^n_xu$, for $n\geq1$, then we need a system coupled in the second equation of \eqref{eq:nonlinear_system} in the form $\partial^n_x (1_{\mathcal{O}} \partial^n_x u)$, this means, the coupling has twice as many derivatives. More details about this kind of problem will be given in Section \ref{SecFinal}.
\end{itemize}
\end{remarks}

Our work is outlined in the following way: Section \ref{Sec2} is devoted to presenting a new Carleman estimate which will be the key to proving the main result of this manuscript. In Section \ref{Sec3}, we show the null controllability results, that is, the linear case (Theorem \ref{thm:control1}) and the nonlinear one (Theorem \ref{thm:control}).  Section \ref{SecFinal}, we present further comments and some open problems that seem to be of interest from the mathematical point of view. Finally,  for
completeness, at the end of this paper, we have an Appendix \ref{Sec1} about the existence of solutions for the systems considered here.

\section{Carleman estimates}\label{Sec2}
In this section, we prove a new Carleman estimate to the fourth order Schr\"odinger operator $\partial^4_xu-\partial^2_xu$. For the sake of simplicity, we will consider the operator $\mathcal{L}=\partial^4_xu$, that is, only with the higher term. So, to derive this new Carleman, first, let us introduce the basic weight function $\eta(x) = (x- x_0)$ with $x_0<0$. Now, for $\lambda>1$ and $\mu >1$ we define the following 
\begin{equation}\label{weight}
    \theta =e^l, \quad \xi(t,x) = \frac{e^{3\mu \eta(x)}}{t(T-t)}\quad \text{and} \quad l(t,x) = \lambda \frac{e^{3\mu \eta(x)}- e^{5\mu ||\eta||_{\infty}}}{t(T-t)}.
\end{equation}
Our result will be derived from a previous result due to Zheng \cite{zheng}, which can be seen as follows.
\begin{proposition}\label{zheng}There exists three constants $\mu_0>1$, $C_0 >0$ and $C>0$ such that for all $\mu >\mu_0$ and for all $\lambda \geq C_0(T+T^2)$, 
\begin{equation}\label{zheng1}
\begin{split}
\iint_{Q_T} &\left( \lambda^7\mu^8\xi^7 \theta^2 |u|^2+\lambda^5\mu^6\xi^5 \theta^2 |u_{x}|^2+\lambda^3\mu^4\xi^3 \theta^2 |u_{xx}|^2+\lambda\mu^2\xi \theta^2 |u_{xxx}|^2 \right) dxdt \\ \leq& C\left(\iint_{Q_T} |\theta Pu|^2 dxdt + \lambda^3\mu^3\int_0^T (\xi^3 \theta^2|u_{xx}|^2)(t,L)   dt + \lambda\mu \int_0^T (\xi \theta^2 |u_{xxx}|^2)(t,L) dt \right),
\end{split}
\end{equation}
where $Pu = i\partial_tu + \partial^4_xu$.
\end{proposition}

With the previous theorem in hand, we are in a position to prove a new Carleman estimate associated with the operator $\mathcal{L}u$. The result is stated in the following way.

\begin{theorem}\label{propcarleman} Let $\omega, \mathcal{O} \subset \Omega$ be open subsets such that $\omega \cap \mathcal{O} \neq \emptyset$. Then, there exists a positive constant $\mu_1$, such that for any $\mu > \mu_1$, one can find two positive constants $\lambda_1$ and $C$ depending on $\lambda$, $\mu$, $\Omega$, $\omega$ such that for any $\lambda > \lambda_1(T+T^2)$ the following estimate for $\varphi$ and $\psi$ of \eqref{adjuntoinversa} holds
\begin{equation}\label{eqcarleman}
\begin{split}
\iint_{Q_T} &\left( \lambda^7\mu^8\xi^7 \theta^2 |\varphi|^2+\lambda^5\mu^6\xi^5 \theta^2 |\varphi_{x}|^2+\lambda^3\mu^4\xi^3 \theta^2 |\varphi_{xx}|^2+\lambda\mu^2\xi \theta^2 |\varphi_{xxx}|^2 \right) dxdt \\ &+\iint_{Q_T} \left( \lambda^7\mu^8\xi^7 \theta^2 |\psi|^2+\lambda^5\mu^6\xi^5 \theta^2 |\psi_{x}|^2+\lambda^3\mu^4\xi^3 \theta^2 |\psi_{xx}|^2+\lambda\mu^2\xi \theta^2 |\psi_{xxx}|^2 \right) dxdt \\  \leq &C\left( \iint_{Q_T} \theta^2(|g^0|^2+|g^1|^2) dxdt +  \lambda \mu \iint_{\omega_T}\xi \theta
^2|\varphi|^2 dxdt\right). 
\end{split}
\end{equation}
\end{theorem}

Before proving to prove this result, let us give the idea to derive \eqref{eqcarleman}. To do it, 
we split the proof into several steps. The first one consists of applying the Carleman estimate given by Proposition \ref{zheng} for $(\varphi,\psi)$ solutions of
\begin{equation}\label{adjuntoinversa}
	\left\{\begin{array}{lll}
      i \varphi_t  +  \varphi_{xx} - \varphi_{xxxx} = 1_{\mathcal{O}}\psi +g^0,  &\text{in}& {Q_T},\\
      i \psi_t  +  \psi_{xx}  - \psi_{xxxx} = g^1,   &\text{in}& {Q_T},\\
      \varphi(t,0) = \varphi(t,L) =  \varphi_x(t,0) = \varphi_x(t,L) = 0, &\text{on}& t \in (0,T), \\  \psi(t,0) = \psi(t,L) = \psi_x(t,0) = \psi_x(t,L) = 0, &\text{on}& t \in (0,T), \\
      \varphi(T,x)=0, \psi(0,x) = \psi_0, &\text{in}&\Omega.
	\end{array} \right.
	\end{equation}
The second step concerns the estimate for a local integral term of $\psi$ in terms of a local integral of $\varphi$ and global integral terms of $g^0$, $g^1$, $\varphi$, $\psi$ and smaller order terms of $\psi$. Finally, we will estimate integral terms on the border in terms of the global integral of $\varphi$, $\psi$, and smaller order integral terms.

\begin{proof}[Proof of Theorem \ref{propcarleman}] In what follows, remember that  $\Omega \subset \mathbb{R}$ is a bounded domain whose boundary $\partial\Omega$ is regular enough.  Consider $T>0$,  $\omega$ and $\mathcal{O}$ to be two nonempty subsets of $\Omega$. Additionally, as defined at the beginning of this work $Q_{T}=\Omega \times (0, T)$, $\omega_{T}=\omega \times (0, T)$, $\Sigma_{T}=\partial \Omega \times(0, T)$, $\mathcal{O}_{T}=\mathcal{O} \times(0, T)$ and denote by $C$ a generic constant which can be different from one computation to another. Thus, let us split the proof into three steps.

\vspace{0.2cm}

\noindent \textbf{Step 1: Applying Carleman estimates \eqref{zheng1}.}

\vspace{0.1cm}

Thanks to \eqref{zheng1} we have, for  $\varphi$ and $\psi$, solution of \eqref{adjuntoinversa}, that
\begin{equation}\label{eq1passo1}
\begin{split}
&\iint_{Q_T} \left( \lambda^7\mu^8\xi^7 \theta^2 |\varphi|^2+\lambda^5\mu^6\xi^5 \theta^2 |\varphi_{x}|^2+\lambda^3\mu^4\xi^3 \theta^2 |\varphi_{xx}|^2+\lambda\mu^2\xi \theta^2 |\varphi_{xxx}|^2 \right) dxdt \\& \quad \quad \leq C\left(\iint_{Q_T} \theta^2 |\varphi_{xx}| dxdt + \iint_{Q_T} \theta^2|g^0|^2 dxdt + \iint_{\mathcal{O}_T} \theta^2|\psi|^2 dxdt \right. \\ & \quad \quad \quad \left.+ \lambda^3\mu^3\int_0^T (\xi^3 \theta^2|\varphi_{xx}|^2)(t,L)   dt + \lambda\mu \int_0^T (\xi \theta^2 |\varphi_{xxx}|^2)(t,L) dt \right). 
\end{split}
\end{equation}
For $\lambda,\mu$ large enough, we obtain
\begin{equation}\label{eq2passo1}
\begin{split}
&\iint_{Q_T} \left( \lambda^7\mu^8\xi^7 \theta^2 |\varphi|^2+\lambda^5\mu^6\xi^5 \theta^2 |\varphi_{x}|^2+\lambda^3\mu^4\xi^3 \theta^2 |\varphi_{xx}|^2+\lambda\mu^2\xi \theta^2 |\varphi_{xxx}|^2 \right) dxdt \\& \quad \quad \leq C\left( \iint_{Q_T} \theta^2|g^0|^2 dxdt + \iint_{\mathcal{O}_T} \theta^2|\psi|^2 dxdt + \lambda^3\mu^3\int_0^T (\xi^3 \theta^2|\varphi_{xx}|^2)(t,L)   dt \right. \\& \quad \quad  \quad \left.+ \lambda\mu \int_0^T (\xi \theta^2 |\varphi_{xxx}|^2)(t,L) dt \right).
\end{split}
\end{equation}
Now, applying \eqref{zheng1} for $\psi$, we get that
\begin{equation}\label{eq3passo1}
\begin{split}
&\iint_{Q_T} \left( \lambda^7\mu^8\xi^7 \theta^2 |\psi|^2+\lambda^5\mu^6\xi^5 \theta^2 |\psi_{x}|^2+\lambda^3\mu^4\xi^3 \theta^2 |\psi_{xx}|^2+\lambda\mu^2\xi \theta^2 |\psi_{xxx}|^2 \right) dxdt \\ & \quad \quad \leq C\left(\iint_{Q_T} \theta^2 |\psi_{xx}| dxdt + \iint_{Q_T} \theta^2|g^1|^2 dxdt + \iint_{\mathcal{O}_T} |\psi|^2 dxdt \right. \\&\quad \quad \quad   \left.+ \lambda^3\mu^3\int_0^T (\xi^3 \theta^2|\varphi_{xx}|^2)(t,L)   dt + \lambda\mu \int_0^T (\xi \theta^2 |\varphi_{xxx}|^2)(t,L) dt \right).
\end{split}
\end{equation}
Again, by using \eqref{eq3passo1} for $\lambda, \mu$ large enough we have the following:
\begin{equation}\label{eq4passo1}
\begin{split}
&\iint_{Q_T} \left( \lambda^7\mu^8\xi^7 \theta^2 |\psi|^2+\lambda^5\mu^6\xi^5 \theta^2 |\psi_{x}|^2+\lambda^3\mu^4\xi^3 \theta^2 |\psi_{xx}|^2+\lambda\mu^2\xi \theta^2 |\psi_{xxx}|^2 \right) dxdt \\ & \leq C\left( \iint_{Q_T} \theta^2|g^1|^2 dxdt  + \lambda^3\mu^3\int_0^T (\xi^3 \theta^2|\psi_{xx}|^2)(t,L)   dt + \lambda\mu \int_0^T (\xi \theta^2 |\psi_{xxx}|^2)(t,L) dt \right).
\end{split}
\end{equation}
Note that, putting together \eqref{eq2passo1} and \eqref{eq4passo1}, we obtain the following estimate
\begin{equation}\label{eq5passo1}
\begin{split}
&\iint_{Q_T} \left( \lambda^7\mu^8\xi^7 \theta^2 |\varphi|^2+\lambda^5\mu^6\xi^5 \theta^2 |\varphi_{x}|^2+\lambda^3\mu^4\xi^3 \theta^2 |\varphi_{xx}|^2+\lambda\mu^2\xi \theta^2 |\varphi_{xxx}|^2 \right) dxdt \\& + \iint_{Q_T} \left( \lambda^7\mu^8\xi^7 \theta^2 |\psi|^2+\lambda^5\mu^6\xi^5 \theta^2 |\psi_{x}|^2+\lambda^3\mu^4\xi^3 \theta^2 |\psi_{xx}|^2+\lambda\mu^2\xi \theta^2 |\psi_{xxx}|^2 \right) dxdt \\ &   \leq C\left( \iint_{Q_T} \theta^2|g^0|^2 dxdt + \iint_{Q_T} \theta^2|g^1|^2 dxdt +\iint_{\mathcal{O}_T} \theta^2|\psi|^2 dxdt \right. \\ &  + \lambda^3\mu^3\int_0^T (\xi^3 \theta^2(|\varphi_{xx}|^2+|\psi_x|^2))(t,L)   dt +\lambda\mu \int_0^T (\xi \theta^2 (|\varphi_{xxx}|^2+|\psi_{xxx}|^2))(t,L) dt .
\end{split}
\end{equation}

\vspace{0.1cm}

\noindent \textbf{Step 2: Estimates for the local integral of $\psi$.} 

\vspace{0.1cm}

In this step, let us estimate the last term in the right-hand side of \eqref{eq5passo1}, that is, the local integral term of $\psi$ in $\mathcal{O}_T$. Now, since $\omega \cap \mathcal{O} \neq \emptyset$, there exists $\tilde{\omega}_T \subset \omega \cap \mathcal{O}$. From now on, take a cut-off function $\eta \in C_0^{\infty}(\omega)$ such that $\eta \equiv 1$ in $\tilde{\omega}_T$. Observe that 
$$\psi = -i\varphi_t + \varphi_{xx} - \varphi_{xxxx} - g^0, \quad \text{in } \mathcal{O}_T,$$
so
\begin{equation}\label{eq1passo2}
\begin{aligned}
    \iint_{\mathcal{O}_T}\theta^2|\psi|^2 dxdt \leq \iint_{\tilde{\omega}_T}\eta \theta^2|\psi|^2 dxdt &= \iint_{\tilde{\omega}_T}\eta \theta^2\overline{\psi}\psi dxdt \\& = Re\iint_{\tilde{\omega}_T}\eta \theta^2\overline{\psi}\left(-i\varphi_t + \varphi_{xx} - \varphi_{xxxx} - g^0\right) dxdt \\&:= \sum_{i=1}^{4} \Psi_i,
    \end{aligned}
\end{equation}
where $\Psi_i$, for $i = 1,2,3,4$, are the integrals of the right-hand side of \eqref{eq1passo2}. We now estimate these terms. For $\Psi_1$, integrating by parts in $t$ we have that

\begin{equation}\label{eq2passo2}
\begin{split}
\Psi_1 &= Re\iint_{\tilde{\omega}_T} \eta \theta^2 \overline{\psi} (-i\varphi_t) dxdt = Re\left( \left(\eta \theta^2 \psi (-i\varphi)\right)\bigg|^T_0 - \iint_{Q_T} \left(\eta \theta^2 \overline{\psi} \right)_t ( -i\varphi)dxdt  \right) \\ &=Re\left(- \iint_{\tilde{\omega}_T} \left(\eta \theta^2\right)_t \overline{\psi} (-i\varphi)dxdt - \iint_{\tilde{\omega}_T} \eta \theta^2 \overline{\psi}_t  (-i\varphi)dxdt\right).
\end{split}
\end{equation}
Since,
$$|(\theta^{m})_t| = |(e^{ml})_t| = |e^{ml}(ml)_t| =|me^{ml}(l)_t| \leq mC\lambda T \xi^2 \theta^m,$$
then, for $m = 2$, we get that $|(\eta \theta^2)_t| \leq C\lambda \xi^2\theta^2$, and by using Young inequality we obtain 
\begin{equation}\label{eq3passo2}
\begin{split}
    Re \left(- \iint_{\tilde{\omega}_T} \left(\eta \theta^2\right)_t \overline{\psi} (-i\varphi)dxdt\right)& \leq CRe \left(\lambda \iint_{\tilde{\omega}_T} \xi^2 \theta^2 \overline{\psi}\varphi dxdt\right) \\&=  Re \left(\iint_{\tilde{\omega}_T} \left(\lambda^{\frac{7}{2}}\mu^4\xi^{\frac{7}{2}} \theta \overline{\psi}\right) \left(C \lambda^{-\frac{5}{2}} \mu^{-4}\xi^{-\frac{3}{2}} \varphi\right) dxdt   \right) \\
   & \leq \delta \lambda^{7}\mu^8\iint_{Q_T} \xi^{7}\theta^2|\psi|^2dxdt +C\lambda^{-5}\mu^{-8}\iint_{\tilde{\omega}_T}\xi^{-3}|\varphi|^2dxdt.
\end{split}
\end{equation}
Combining \eqref{eq2passo2} with \eqref{eq3passo2}, we get
\begin{equation}\label{eq4passo2}
\begin{split}
&\Psi_1  =Re\left(- \iint_{\tilde{\omega}_T} \eta \theta^2  \overline{\psi}_t  (-i\varphi)dxdt\right) \\ &  \quad \quad + \delta \lambda^{7}\mu^8\iint_{Q_T} \xi^{7}\theta^2|\psi|^2dxdt +C\lambda^{-5}\mu^{-8}\iint_{\tilde{\omega}_T}\xi^{-3}|\varphi|^2dxdt .
\end{split}
\end{equation}
So, for $\delta$ small enough we can absorb the global integral term of \eqref{eq4passo2} with the left-hand side of \eqref{eq5passo1}. Due to the boundary conditions for $\psi$ and $\varphi$ and since $\eta$ has compact support on $\tilde{\omega}$ by integrating by parts for space variable, we obtain
\begin{equation}\label{eq5passo2}
\begin{split}
   \Psi_2 &= Re\iint_{\tilde{\omega}_T} \eta \theta^2\overline{\psi} \varphi_{xx} dxdt = Re\left( \iint_{\tilde{\omega}_T}  \left(\eta \theta^2 \overline{\psi}\right)_{xx} \varphi dxdt \right) \\ &= Re\left( \iint_{\tilde{\omega}_T}  \left(\eta \theta^2\right)_{xx} \overline{\psi} \varphi dxdt + 2\iint_{\tilde{\omega}_T} \left (\eta \theta^2\right)_{x} \overline{\psi}_x \varphi dxdt+\iint_{\tilde{\omega}_T}  \eta \theta^2 \overline{\psi}_{xx} \varphi dxdt \right)  \\ &= \sum_{i=1}
  ^{3}\Psi_{2,i}.
\end{split}
\end{equation}
To bound each term of the right-hand side of \eqref{eq5passo2} first observe that the following estimates, for the derivative in space of the weight function $\theta$, hold true
$$\left|\left(\theta^2\right)_x\right| = \left|\left(e^{2l}\right)_x\right| = \left|e^{2l}(2l)_x\right| = 2\left|\theta^2 l_x\right| \leq C\lambda \mu \xi\theta^2$$
and 
\begin{equation*}
\begin{split}
\left|\left(\theta^2\right)_{xx}\right| &= \left|\left(e^{2l}\right)_{xx}\right| = \left|\left(e^{2l}(2l)_x\right)_x\right| = 2\left|\left(\theta^2 l_x\right)_x\right| \\ &\leq 4\left|\theta^2 l_x^2\right| +2\left|\theta^2 l_{xx}\right| \leq 4\theta^2|l_x|^2 +2\theta^2|l_{xx}|
\\&\leq C\lambda^2\mu^2\xi^2\theta^2 +C\lambda \mu \xi^2\theta^2 \leq C \lambda^2 \mu^2 \xi^2\theta^2.
\end{split}
\end{equation*}
Therefore, it yields that
\begin{equation}\label{eq6passo2}
\begin{split}
    \Psi_{2,1} &= Re\left(\iint_{\tilde{\omega}_T} \left(\eta \theta^2\right)_{xx}\overline{\psi}\varphi dxdt \right) \leq Re\left(C \lambda^2 \mu^2\iint_{\tilde{\omega}_T}  \xi^2\theta^2 \overline{\psi}\varphi dxdt \right) \\ &= Re \left(\iint_{\tilde{\omega}_T}  \left(\lambda^{\frac{7}{2}} \mu^{4}\xi^{\frac{7}{2}}\theta \overline{\psi}\right) \left(C\lambda^{-\frac{3}{2}} \mu^{-2} \xi^{-\frac{3}{2}} \theta \varphi\right) dxdt  \right) \\ &\leq \delta\lambda^{7} \mu^{8} \iint_{Q_T}  \xi^{7}\theta^2 |\psi|^2dxdt +C\lambda^{-3} \mu^{-4} \iint_{\tilde{\omega}_T}\xi^{-3} \theta^2|\varphi|^2 dxdt
\end{split}
\end{equation}
and
\begin{equation}\label{eq7passo2}
\begin{split}
    \Psi_{2,2} &= 2 Re\left(\iint_{\tilde{\omega}_T} \left(\eta \theta^2\right)_{x}\overline{\psi}_x\varphi dxdt \right) \leq Re\left(C \lambda \mu\iint_{\tilde{\omega}_T}  \xi\theta^2 \overline{\psi}\varphi dxdt \right) \\&= 2Re \left(\iint_{\tilde{\omega}_T}  \left(\lambda^{\frac{5}{2}} \mu^{3}\xi^{\frac{5}{2}}\theta \overline{\psi}_x\right) \left(C\lambda^{-\frac{3}{2}} \mu^{-2} \xi^{-\frac{3}{2}} \theta \varphi\right) dxdt  \right) \\ &\leq  \delta\lambda^{5} \mu^{6} \iint_{Q_T}  \xi^{5}\theta^2 |\psi_x|^2dxdt +C\lambda^{-3} \mu^{-4} \iint_{\tilde{\omega}_T}\xi^{-3} \theta^2|\varphi|^2 dxdt.
\end{split}
\end{equation}
Finally, $\Psi_{2,3}$ does not need to be estimated since we use it to obtain the equation for $\psi$. Now, combining \eqref{eq5passo2}, \eqref{eq6passo2} and \eqref{eq7passo2} we get 
\begin{equation}\label{eq8passo2}
\begin{split}
    \Psi_2 &\leq C\lambda^{-3}\mu^{-4} \iint_{\tilde{\omega}_T} \xi^{-3}\theta^2|\varphi|^2 dxdt \\
    &+\delta \left(\lambda^{7} \mu^{8} \iint_{Q_T}  \xi^{7}\theta^2 |\psi|^2dxdt+\lambda^{7} \mu^{8} \iint_{Q_T}  \xi^{7}\theta^2 |\psi|^2dxdt+\lambda^{7} \mu^{8} \iint_{Q_T}  \xi^{7}\theta^2 |\psi_{xx}|^2dxdt \right).
\end{split}
\end{equation}

For $\Psi_3$ we use the boundary conditions and integrate with respect to the space variable four times to obtain
\begin{equation}\label{eq9passo2}
\begin{split}
   \Psi_3 =& Re\iint_{\tilde{\omega}_T} \eta \theta^2\overline{\psi} (-\varphi_{xxxx}) dxdt = Re\left( \iint_{\tilde{\omega}_T}  \left(\eta \theta^2 \overline{\psi}\right)_{xxxx} (-\varphi) dxdt \right) \\
  =& Re\left( \iint_{\tilde{\omega}_T}  \left(\eta \theta^2\right)_{xxxx} \overline{\psi} \varphi dxdt + 4\iint_{\tilde{\omega}_T}  \left(\eta \theta^2\right)_{xxx} \overline{\psi}_x \varphi dxdt +6\iint_{\tilde{\omega}_T}  \left(\eta \theta^2\right)_{xx} \overline{\psi}_{xx}dxdt \right) \\& +Re\left(4\iint_{\tilde{\omega}_T}  \left(\eta \theta^2\right)_{x} \overline{\psi}_{xxx}  \varphi dxdt   +\iint_{\tilde{\omega}_T}  \eta \theta^2 \left(-\overline{\psi}_{xxxx}\right) \varphi dxdt \right)  \\
  =& \sum_{i=1}
  ^{5}\Psi_{3,i}.
\end{split}
\end{equation}
Now our task is to estimate these terms. Observe that we have the following estimates in space variable for $k-$th order derivative in space  variable for the weight function $\theta$, $$\left|\left(\theta^2\right)_{kx}\right| \leq 2^k C \left(\lambda^k \mu^k \xi^k\right)\theta^2,$$
thus

\begin{equation}\label{eq10passo2}
\begin{split}
\Psi_{3,1} &=    Re \left( \iint_{\tilde{\omega}_T}  \left(\eta \theta^2\right)_{xxxx} \overline{\psi} \varphi dxdt \right) \\&\leq Re \left( C\lambda^4\mu^4 \iint_{\tilde{\omega}_T}  \xi^4\theta^2  \overline{\psi} \varphi dxdt \right) \\& =  Re \left(\iint_{\tilde{\omega}_T}  \left(\lambda^{\frac{7}{2}} \mu^4 \xi^{\frac{7}{2}}\theta \overline{\psi}\right) \left(C\lambda^{\frac{1}{2}}  \xi^{\frac{1}{2}} \theta \varphi\right) dxdt \right) \\ &\leq  \delta\lambda^{7} \mu^{8} \iint_{Q_T}  \xi^{7}\theta^2 |\psi|^2dxdt +C\lambda  \iint_{\tilde{\omega}_T}\xi \theta^2|\varphi|^2 dxdt,
\end{split}
\end{equation}

\begin{equation}\label{eq11passo2}
\begin{split}
\Psi_{3,2} &=     Re \left( 4\iint_{\tilde{\omega}_T}  \left(\eta \theta^2\right)_{xxx} \overline{\psi}_x \varphi dxdt \right)\\& \leq Re \lambda^3 \mu^3\left( \iint_{\tilde{\omega}_T}  \xi^3 \theta^2 \overline{\psi}_x \varphi dxdt \right) \\ &=  Re \left(\iint_{\tilde{\omega}_T}  \left(\lambda^{\frac{5}{2}} \mu^3 \xi^{\frac{5}{2}}\theta \overline{\psi}_x\right) \left(C\lambda^{\frac{1}{2}}  \xi^{\frac{1}{2}} \theta \varphi\right) dxdt \right) \\ &\leq  \delta\lambda^{5} \mu^{6} \iint_{Q_T}  \xi^{5}\theta^2 |\psi_{x}|^2dxdt +C\lambda  \iint_{\tilde{\omega}_T}\xi \theta^2|\varphi|^2 dxdt,
\end{split}
\end{equation}

\begin{equation}\label{eq12passo2}
\begin{split}
\Psi_{3,3} &=     Re \left( 6\iint_{\tilde{\omega}_T}  \left(\eta \theta^2\right)_{xx} \overline{\psi}_{xx}dxdt \right)\\& \leq   Re \left(C\lambda^2\mu^2 \iint_{\tilde{\omega}_T}  \xi^2 \theta^2 \overline{\psi}_{xx}dxdt \right) \\ &=  Re \left(\iint_{\tilde{\omega}_T}  \left(\lambda^{\frac{3}{2}} \mu^2 \xi^{\frac{3}{2}}\theta \overline{\psi}_{xx}\right) \left(C\lambda^{\frac{1}{2}}  \xi^{\frac{1}{2}} \theta \varphi\right) dxdt \right)\\ &\leq  \delta\lambda^{3} \mu^{4} \iint_{Q_T}  \xi^{3}\theta^2 |\psi_{xx}|^2dxdt +C\lambda  \iint_{\tilde{\omega}_T}\xi \theta^2|\varphi|^2 dxdt
\end{split}
\end{equation}
and
\begin{equation}\label{eq13passo2}
\begin{split}
\Psi_{3,4} &=     Re \left( 4\iint_{\tilde{\omega}_T}  \left(\eta \theta^2\right)_{x} \overline{\psi}_{xxx}  \varphi dxdt \right) \\&\leq Re \left(C \lambda \mu\iint_{\tilde{\omega}_T}\xi  \theta^2 \overline{\psi}_{xxx}  \varphi dxdt \right) \\& =  Re \left(\iint_{\tilde{\omega}_T}  \left(\lambda^{\frac{1}{2}} \mu \xi^{\frac{1}{2}}\theta \overline{\psi}_{xxx}\right) \left(C\lambda^{\frac{1}{2}}  \xi^{\frac{1}{2}} \theta \varphi\right) dxdt \right)\\ &\leq  \delta\lambda \mu^{2} \iint_{Q_T}  \xi\theta^2 |\psi_{xxx}|^2dxdt +C\lambda  \iint_{\tilde{\omega}_T}\xi \theta^2|\varphi|^2 dxdt. 
\end{split}
\end{equation}
We do not estimate $\Psi_{3,5}$ since we use this term to obtain the equation for $\psi$. By putting  \eqref{eq10passo2}, \eqref{eq11passo2}, \eqref{eq12passo2} and  \eqref{eq13passo2} in \eqref{eq9passo2}, we conclude 
\begin{equation}\label{eq14passo2}
\begin{split}
    \Psi_3 &\leq  Re\left(\iint_{\tilde{\omega}_T}  \eta \theta^2 \overline{\psi}_{xxxx} \varphi dxdt \right) + C\lambda  \iint_{\tilde{\omega}_T}\xi \theta^2|\varphi|^2 dxdt \\
    &+\delta \left(
 \iint_{Q_T} \left[\lambda^{7}\mu^{8}\xi^{7}\theta^2 |\psi|^2+
\lambda^{5}\mu^{6}\xi^{5}\theta^2 |\psi_{x}|^2+
\lambda^{3}\mu^{4}\xi^{3}\theta^2 |\psi_{xx}|^2+
\lambda\mu^{2}\xi\theta^2 |\psi_{xxx}|^2 \right]dxdt \right).
\end{split}
\end{equation}

Finally, for $\Psi_4$ we get 
\begin{equation}\label{eq15passo2}
\begin{split}
\Psi_4 &= Re \iint_{\tilde{\omega}_T} \eta \theta^2 \overline{\psi} \left(-g^0\right) dxdt \\&\leq \delta \lambda^7 \mu^8 \iint_{Q_T} \xi^7 \theta^2 |\psi|^2 dxdt + C\lambda^{-7} \mu^{-8} \iint_{Q_T} \xi^{-7} \theta^2 |g^0|^2 dxdt.
\end{split}
\end{equation}

Combining \eqref{eq1passo2}, \eqref{eq4passo2}, \eqref{eq8passo2}, \eqref{eq14passo2} and \eqref{eq15passo2} we get
\begin{equation}\label{eq16passo2}
\begin{split}
&    \iint_{\mathcal{O}_T} \theta^2|\psi|^2 dxdt     \\ &\leq C\left(\lambda  \iint_{\tilde{\omega}_T}\xi \theta^2|\varphi|^2 dxdt +\lambda^{-3}\mu^{-4} \iint_{\tilde{\omega}_T} \xi^{-3}\theta^2|\varphi|^2 dxdt+\lambda^{-5}\mu^{-8}\iint_{\tilde{\omega}_T}\xi^{-3}|\varphi|^2dxdt \right) \\ &+Re \left( \iint_{\tilde{\omega}_T} \eta \theta^2 \varphi \left(i\overline{\psi}_t + \overline{\psi}_{xx} - \overline{\psi}_{xxxx}\right) dxdt \right) + C\lambda^{-7} \mu^{-8} \iint_{Q_T} \xi^{-7} \theta^2 |g^0|^2 dxdt \\ &+ 
    \delta \left(
 \iint_{Q_T} \left[\lambda^{7}\mu^{8}\xi^{7}\theta^2 |\psi|^2+
\lambda^{5}\mu^{6}\xi^{5}\theta^2 |\psi_{x}|^2+
\lambda^{3}\mu^{4}\xi^{3}\theta^2 |\psi_{xx}|^2+
\lambda\mu^{2}\xi\theta^2 |\psi_{xxx}|^2 \right]dxdt \right).
\end{split}
\end{equation}
Since we get $i\overline{\psi}_t + \overline{\psi}_{xx} - \overline{\psi}_{xxxx} = \overline{-i\psi_t + \psi_{xx} - \psi_{xxxx}} = \overline{g^1}$  in $Q_T$  and 
\begin{equation}\label{eq17passo2}
\begin{split}
    Re \left( \iint_{\tilde{\omega}_T} \eta \theta^2 \varphi \overline{g^1} dxdt \right) \leq \delta \lambda^7 \mu^8 \iint_{Q_T} \xi^7 \theta^2 |\varphi|^2 dxdt + C\lambda^{-7} \mu^{-8} \iint_{Q_T} \xi^{-7} \theta^2 |g^1|^2 dxdt,
    \end{split}
\end{equation}
then, putting together \eqref{eq16passo2} and \eqref{eq17passo2}, yields that
\begin{equation}\label{eq18passo2}
\begin{split}
    &\iint_{\mathcal{O}_T}\theta^2|\psi|^2 dxdt\leq C\lambda^{-7} \mu^{-8} \iint_{Q_T} \xi^{-7} \theta^2 \left(|g^0|^2+ |g^1|^2\right) dxdt + \delta \lambda^7 \mu^8 \iint_{Q_T} \xi^7 \theta^2 |\varphi|^2 dxdt   \\ &+ C\left(\lambda  \iint_{\tilde{\omega}_T}\xi \theta^2|\varphi|^2 dxdt +\lambda^{-3}\mu^{-4} \iint_{\tilde{\omega}_T} \xi^{-3}\theta^2|\varphi|^2 dxdt+\lambda^{-5}\mu^{-8}\iint_{\tilde{\omega}_T}\xi^{-3}|\varphi|^2dxdt \right) \\ &+ 
    \delta \left(
 \iint_{Q_T} \left[\lambda^{7}\mu^{8}\xi^{7}\theta^2 |\psi|^2+
\lambda^{5}\mu^{6}\xi^{5}\theta^2 |\psi_{x}|^2+
\lambda^{3}\mu^{4}\xi^{3}\theta^2 |\psi_{xx}|^2+
\lambda\mu^{2}\xi\theta^2 |\psi_{xxx}|^2 \right]dxdt \right). 
\end{split}
\end{equation}

Combining \eqref{eq5passo1} with \eqref{eq18passo2} we obtain the following 
\begin{equation*}
\begin{split}
\iint_{Q_T}& \left( \lambda^7\mu^8\xi^7 \theta^2 |\varphi|^2+\lambda^5\mu^6\xi^5 \theta^2 |\varphi_{x}|^2+\lambda^3\mu^4\xi^3 \theta^2 |\varphi_{xx}|^2+\lambda\mu^2\xi \theta^2 |\varphi_{xxx}|^2 \right) dxdt \\& +\iint_{Q_T} \left( \lambda^7\mu^8\xi^7 \theta^2 |\psi|^2+\lambda^5\mu^6\xi^5 \theta^2 |\psi_{x}|^2+\lambda^3\mu^4\xi^3 \theta^2 |\psi_{xx}|^2+\lambda\mu^2\xi \theta^2 |\psi_{xxx}|^2 \right) dxdt \\ \leq&  \ C\left( \iint_{Q_T} \theta^2\left(|g^0|^2+|g^1|^2\right) dxdt +\lambda^{-7} \mu^{-8} \iint_{Q_T} \xi^{-7} \theta^2 \left(|g^0|^2+ |g^1|^2\right) dxdt \right) \\
&+C\left(\lambda  \iint_{\tilde{\omega}_T}\xi \theta^2|\varphi|^2 dxdt +\lambda^{-3}\mu^{-4} \iint_{\tilde{\omega}_T} \xi^{-3}\theta^2|\varphi|^2 dxdt+\lambda^{-5}\mu^{-8}\iint_{\tilde{\omega}_T}\xi^{-3}|\varphi|^2dxdt \right)  \\
&    +\delta \left(
 \iint_{Q_T}\left[\lambda^{7}\mu^{8}\xi^{7}\theta^2 (|\varphi|^2 +|\psi|^2)+
\lambda^{5}\mu^{6}\xi^{5}\theta^2 |\psi_{x}|^2+
\lambda^{3}\mu^{4}\xi^{3}\theta^2 |\psi_{xx}|^2+
\lambda\mu^{2}\xi\theta^2 |\psi_{xxx}|^2 \right]dxdt \right) \\
&+C\left(\lambda^3\mu^3\int_0^T \left(\xi^3\theta^2\left(|\varphi_{xx}|^2+|\psi_{xx}|^2\right)\right)(t,L)   dt + \lambda\mu \int_0^T \left(\xi \theta^2\left(|\varphi_{xxx}|^2+ |\psi_{xxx}|^2\right)\right)(t,L) dt \right).
\end{split}
\end{equation*}
Then, for $\lambda, \mu$ large enough and $\delta$ small enough we get
\begin{equation}\label{eq19passo2}
\begin{split}
&\iint_{Q_T} \left( \lambda^7\mu^8\xi^7 \theta^2 |\varphi|^2+\lambda^5\mu^6\xi^5 \theta^2 |\varphi_{x}|^2+\lambda^3\mu^4\xi^3 \theta^2 |\varphi_{xx}|^2+\lambda\mu^2\xi \theta^2 |\varphi_{xxx}|^2 \right) dxdt \\& +\iint_{Q_T} \left( \lambda^7\mu^8\xi^7 \theta^2 |\psi|^2+\lambda^5\mu^6\xi^5 \theta^2 |\psi_{x}|^2+\lambda^3\mu^4\xi^3 \theta^2 |\psi_{xx}|^2+\lambda\mu^2\xi \theta^2 |\psi_{xxx}|^2 \right) dxdt \\ &\leq  C\left( \iint_{Q_T} \theta^2\left(|g^0|^2+|g^1|^2\right) dxdt +\lambda  \iint_{\tilde{\omega}_T}\xi \theta^2|\varphi|^2 dxdt\right) \\
&+C\lambda^3\mu^3\int_0^T (\xi^3 \theta^2(|\varphi_{xx}|^2+|\psi_{xx}|^2))(t,L)   dt + C\lambda\mu \int_0^T \left(\xi \theta^2\left(|\varphi_{xxx}|^2+ |\psi_{xxx}|^2\right)\right)(t,L) dt \\
&=:\mathcal{I}+\mathcal{B}_1+\mathcal{B}_2.
\end{split}
\end{equation}

\vspace{0.1cm}

\noindent \textbf{Step 3: Estimates of the boundary terms.} 

\vspace{0.1cm}
Now, we will find an estimate for the boundary term on the right-hand side of \eqref{eq19passo2}, precisely, $\mathcal{B}_1$ and $\mathcal{B}_2$. Using trace Theorem (note that $5/2<7/2$) we have
\begin{equation}\label{eq1passo3}
\begin{split}
\mathcal{B}_1=&\ C\lambda^3\mu^3\int_0^T \left(\xi^3 \theta^2 \left(|\varphi_{xx}|^2+|\psi_{xx}|^2\right)\right)(t,L)   dt \\
\leq & \ C\lambda^3\mu^3\int_0^T\left(\xi^3 \theta^2\right)(t,L)\left(\left\| \varphi\right\|_{H^{\frac{5}{2}}(\tilde{\omega}_T)}^2+\left\| \psi\right\|_{H^{\frac{5}{2}}(\tilde{\omega}_T)}^2\right)dt\\
\leq & \ C\lambda^3\mu^3\int_0^T\left(\xi^3 \theta^2\right)(t,L)\left(\left\| \varphi\right\|_{H^{\frac{7}{2}}\left(\tilde{\omega}_T\right)}^2+\left\| \psi\right\|_{H^{\frac{7}{2}}(\tilde{\omega}_T)}^2\right)dt
\end{split}
\end{equation}
and
\begin{equation}\label{eq2passo3}
\begin{split}
\mathcal{B}_2=&\ C\lambda\mu\int_0^T \left(\xi \theta^2\left(|\varphi_{xxx}|^2+|\psi_{xxx}|^2\right)\right)(t,L)   dt \\
\leq & \ C\lambda^3\mu^3\int_0^T\left(\xi^3 \theta^2\right)(t,L)\left(\left\| \varphi\right\|_{H^{\frac{7}{2}}(\tilde{\omega}_T)}^2+\left\| \psi\right\|_{H^{\frac{7}{2}}(\tilde{\omega}_T)}^2\right)dt.
\end{split}
\end{equation}
So, putting together \eqref{eq1passo3} and \eqref{eq2passo3}, yields that
\begin{equation}\label{eq3passo3}
\begin{split}
\mathcal{B}_1+\mathcal{B}_2\leq & \ C\lambda^3\mu^3\int_0^T(\xi^3\theta^2)(t,L)\left(\left\| \varphi\right\|_{H^{\frac{7}{2}}(\tilde{\omega}_T)}^2+\left\| \psi\right\|_{H^{\frac{7}{2}}(\tilde{\omega}_T)}^2\right)dt.
\end{split}
\end{equation}

Using interpolation in the Sobolev spaces $H^s(\Omega)$, for $s\geq0$, yields that
\begin{equation*}
\begin{split}
\mathcal{B}_1+\mathcal{B}_2\leq & \ C_1\lambda^3\mu^3\int_0^T\left(\xi^3 \theta^2\right)(t,L)\left\| \varphi\right\|_{H^{\frac{11}{3}}(\tilde{\omega}_T)}^{\frac{21}{11}}\left\| \varphi\right\|^{\frac{1}{11}}_{L^2(\tilde{\omega}_T)}dt\\ &+C_1\lambda^3\mu^3\int_0^T\left(\xi^3 \theta^2\right)(t,L)\left\| \psi\right\|_{H^{\frac{11}{3}}(\tilde{\omega}_T)}^{\frac{21}{11}}\left\| \psi\right\|^{\frac{1}{11}}_{L^2(\tilde{\omega}_T)}dt\\
=&\ C_1\lambda^3\mu^3\int_0^T\left(\xi^{\frac{255}{22}}\xi^{-\frac{189}{22}} \theta^{\frac{86}{22}}\theta^{-\frac{42}{22}}\right)(t,L)\left\| \varphi\right\|_{H^{\frac{11}{3}}(\tilde{\omega}_T)}^{\frac{21}{11}}\left\| \varphi\right\|^{\frac{1}{11}}_{L^2(\tilde{\omega}_T)}dt\\ &+C_1\lambda^3\mu^3\int_0^T\left(\xi^{\frac{255}{22}}\xi^{-\frac{189}{22}} \theta^{\frac{86}{22}}\theta^{-\frac{42}{22}}\right)(t,L)\left\| \psi\right\|_{H^{\frac{11}{3}}(\tilde{\omega}_T)}^{\frac{21}{11}}\left\| \psi\right\|^{\frac{1}{11}}_{L^2(\tilde{\omega}_T)}dt
\\
\leq &\ C_{\epsilon}\lambda^6\mu^6\int_0^T(\xi^{255} \theta^{86})(t,L)\left\| \varphi\right\|^2_{L^2(\tilde{\omega}_T)}dt+\epsilon\lambda^{-2}\mu^{-2}\int_0^T\left(\xi^{-\frac{189}{21}} \theta^{-2}\right)(t,L)\left\| \varphi\right\|_{H^{\frac{11}{3}}(\tilde{\omega}_T)}^2dt
\\
&+C_{\epsilon}\lambda^6\mu^6\int_0^T\left(\xi^{255} \theta^{86}\right)(t,L)\left\| \psi\right\|^2_{L^2(\tilde{\omega}_T)}dt+\epsilon\lambda^{-2}\mu^{-2}\int_0^T\left(\xi^{-\frac{189}{21}} \theta^{-2}\right)(t,L)\left\| \psi\right\|_{H^{\frac{11}{3}}(\tilde{\omega}_T)}^2dt,
\end{split}
\end{equation*}
or equivalently, 
\begin{equation}\label{eq3passo4}
\begin{split}
\mathcal{B}_1+\mathcal{B}_2\leq &\ C_{\epsilon}\lambda^6\mu^6\int_0^T\left(\xi^{255} \theta^{86}\right)(t,L)\left\| \varphi\right\|^2_{L^2(\tilde{\omega}_T)}dt\\&+\epsilon\lambda^{-2}\mu^{-2}\int_0^T\left(\xi^{-9} \theta^{-2}\right)(t,L)\left\| \varphi\right\|_{H^{\frac{11}{3}}(\tilde{\omega}_T)}^2dt\\&+C_{\epsilon}\lambda^6\mu^6\int_0^T\left(\xi^{255} \theta^{86}\right)(t,L)\left\| \psi\right\|^2_{L^2(\tilde{\omega}_T)}dt\\&+\epsilon\lambda^{-2}\mu^{-2}\int_0^T\left(\xi^{-9} \theta^{-2}\right)(t,L)\left\| \psi\right\|_{H^{\frac{11}{3}}(\tilde{\omega}_T)}^2dt
\\
=:&\ \mathcal{I}_1+\mathcal{I}_2+\mathcal{I}_3+\mathcal{I}_4,
\end{split}
\end{equation}
for some positive constant $C_{\epsilon}$.

At this moment, our goal is to prove integrals $\mathcal{I}_i$, for $i=1,2,3$, can be absolved by the left-hand side of \eqref{eq19passo2}.  Let us start with the analysis of $\mathcal{I}_2$, precisely the quantity $$\int_0^T\left(\xi^{-9} \theta^{-2}\right)(t,L)\left\| \varphi\right\|_{H^{\frac{11}{3}}(\tilde{\omega}_T)}^2dt.$$

Consider $\varphi_{1}(  t,x)  :=\xi_{1}(  t) \varphi(
t,x)  $ with
\[
\xi_{1}(  t)  =\theta^{-1}\xi^{-\frac{1}{2}}\text{.}%
\]
Then $\varphi_{1}$ satisfies the system
\begin{equation}
\left\{
\begin{array}
[c]{lll}%
-i\varphi_{1t}+\varphi_{1xx}- \varphi_{1xxxx}=f_{1}:=\xi_{1t}\varphi, &  & \text{in }Q_T  \text{,}\\
\varphi_{1}(  t,0)  =\varphi_{1}(  t,L)  =\varphi_{1x}(  t,0)=\varphi_{1x}(  t,L)=0, &  & \text{on }(  0,T)  \text{,}\\
\varphi_{1}(  T,x)  =0, 
&  & \text{in }\Omega.
\end{array}
\right.  \label{eq4passo4}%
\end{equation}
Now, observe that, since $\varphi_x(t,0)=0$ and $\left\vert \xi_{1t}\right\vert
\leq C \lambda \xi^{\frac{3}{2}}\theta^{-1}$, we have%
\begin{equation}\label{eq5passo4}
\begin{split}
\left\Vert f_{1}\right\Vert _{L^{2}( Q_T  )  }^{2}
\leq&\ C\iint_{Q_T}\theta^{-2} \lambda^2\xi^{3}|\varphi|^{2}dxdt \\
\leq&
\ C \iint _{Q_T} \left\{  \lambda^2 \xi ^{3}  |\varphi| ^{2}+ \lambda^3 |\varphi_{x}|^{2}+\lambda  |\varphi_{xx}|^{2}+\lambda^{-1}  |\varphi_{xxx}|^{2}\right\}  \theta^{-2} dxdt, 
\end{split}
\end{equation}
for some constant $C>0$ and all $s\ge s_0$.
Moreover, thanks to Section \ref{Sec1}, $\varphi_{1}\in L^2(0,T;H^2(\Omega))\cap C([0,T];L^2(\Omega))$. Then, interpolating between 
$L^{2}(  0,T;H^{2}( \Omega)  )  $ and $L^{\infty}(
0,T;L^{2}(  \Omega)  )  $, we infer that $\varphi_{1}\in L^{2}(
0,T;H^{5/3}(  \Omega)  )  $ and%
\begin{equation}
\left\Vert \varphi_{1}\right\Vert _{L^{2}(  0,T;H^{5/3}(  \Omega)
)  }\leq C\left\Vert f_{1}\right\Vert _{L^{2}( Q_T )  }\text{.} \label{eq6passo4}%
\end{equation}

Let $\varphi_{2}(  t,x)  :=\xi_{2}(  t)\varphi(
t,x)  $ with
\[
\xi_{2}=\theta^{-1}\xi^{-\frac{5}{2}%
}\text{.}%
\]
Then $v_{2}$ satisfies system \eqref{eq4passo4} with $f_{1}$ replaced by%
\[
f_{2}:=\xi_{2t}\xi_{1}^{-1}\varphi_{1}\text{.}%
\]

Observe that $\left\vert \theta_{2t}\theta_{1}^{-1}\right\vert \leq C\lambda$. Thus,
we obtain
\begin{equation}
\left\Vert f_{2}\right\Vert _{L^{2}(  0,T;H^{5/3}(  \Omega)
)  }\leq C\lambda \left\Vert \varphi_{1}\right\Vert _{L^{2}(  0,T;H^{5/2}(
\Omega)  )  }\text{.} \label{eq7passo4}%
\end{equation}
Now, by using that $\varphi_{2}$ belongs to $L^2\left(0,T;H^4(\Omega)\right)$ and $L^{\infty}\left(0,T;H^2(\Omega)\right)$,  thanks to \eqref{regularidadeeqlinearizado}, and interpolating  these two spaces, we have that $$\varphi_{2}\in L^{2}(
0,T;H^{11/3}(  \Omega)  )  \cap L^\infty (  0,T;H^{8/3}  (  \Omega)  )  $$ with
\begin{equation}
\left\Vert \varphi_{2}\right\Vert _{L^{2}(  0,T;H^{11/3}(  \Omega)
)  \cap L^{\infty}(  0,T;H^{8/3}(  \Omega)  )  }\leq
C\left\Vert f_{2}\right\Vert _{L^{2}(  0,T;H^{5/3}(  \Omega)
)  }\text{.} \label{eq8passo4}%
\end{equation}
Thus we infer from \eqref{eq6passo4}--\eqref{eq8passo4}, the following
\begin{equation}\label{eq9passo4}
\begin{split}
\left\Vert \varphi_{2}\right\Vert ^2_{L^{2}(  0,T;H^{11/3}( \Omega) )  }
\leq&\  C_1\lambda || f_1||^2_{ L^2( Q_T ) }  \\
\leq&\  C_{2} {\displaystyle\int_{Q_T}}  
\left( \lambda^{3}\xi^{3}  |\varphi| ^{2}+ 
\lambda^{4} |\varphi_{x}|^{2}+\lambda^2  |\varphi_{xx}|^{2}+|\varphi_{xxx}|^{2}\right)  \theta^{-2} dxdt.
\end{split}
\end{equation}
Hence, replacing $\varphi_2 =\theta^{-1}\varphi^{-\frac{9}{2}}$ in \eqref{eq9passo4}, for some constant $C_3>0$,  yields that
\begin{equation}\label{eq10passo4}
\begin{split}
\int_{0}^{T} (\xi^{-9}\theta^{-2})(t,L)&\left\Vert
\varphi(  t,\cdot)  \right\Vert _{H^{11/3}(  \tilde{\omega}_T)  } ^{2}dt\\
\leq& \  C_{3} \int_{Q_T} \left( \lambda^{3}\xi^{3}  |\varphi| ^{2}+ 
\lambda^{4} |\varphi_{x}|^{2}+\lambda^2  |\varphi_{xx}|^{2}+|\varphi_{xxx}|^{2}\right)  \theta^{-2} dxdt.
\end{split}
\end{equation}
Note that analogously we can infer the same relation for $\psi$, that is, 
\begin{equation}\label{eq11passo4}
\begin{split}
\int_{0}^{T}(\xi^{-9} \theta^{-2})(t,L)&\left\Vert
\psi(  t,\cdot)  \right\Vert _{H^{11/3}( \tilde{\omega}_T)  } ^{2}dt\\
\leq&  \ C_{3} \int_{Q_T} \left( \lambda^{3}\xi^{3}  |\psi| ^{2}+ 
\lambda^{4} |\psi_{x}|^{2}+\lambda^2  |\psi_{xx}|^{2}+|\psi_{xxx}|^{2}\right)  \theta^{-2} dxdt.
\end{split}
\end{equation}
Therefore, adding  \eqref{eq10passo4} and \eqref{eq11passo4}, putting in \eqref{eq3passo4} and, finally, comparing with \eqref{eq19passo2}, for $\lambda$ and $\mu$ large enough, yields that,   
\begin{equation}\label{eq12passo4}
\begin{split}
&\iint_{Q_T} \left( \lambda^7\mu^8\xi^7 \theta^2 |\varphi|^2+\lambda^5\mu^6\xi^5 \theta^2 |\varphi_{x}|^2+\lambda^3\mu^4\xi^3 \theta^2 |\varphi_{xx}|^2+\lambda\mu^2\xi \theta^2 |\varphi_{xxx}|^2 \right) dxdt \\& +\iint_{Q_T} \left( \lambda^7\mu^8\xi^7 \theta^2 |\psi|^2+\lambda^5\mu^6\xi^5 \theta^2 |\psi_{x}|^2+\lambda^3\mu^4\xi^3 \theta^2 |\psi_{xx}|^2+\lambda\mu^2\xi \theta^2 |\psi_{xxx}|^2 \right) dxdt \\& \quad \quad\leq  C\left( \iint_{Q_T} \theta^2(|g^0|^2+|g^1|^2) dxdt +\lambda  \iint_{\tilde{\omega}_T}\xi \theta^2|\varphi|^2 dxdt\right),
\end{split}
\end{equation}
since $|(\xi^{255} \theta^{86})(t,L)|$ is bounded in terms of $t\in[0,T]$ due to the choices in \eqref{weight}, what guarantees \eqref{eqcarleman}, and so the Carleman is shown.
\end{proof}

\section{Null controllability results}\label{Sec3}

In this section, we prove the existence of insensitizing controls for the linearized system \eqref{eq:linearized_systeminversa}. First, we need to obtain an estimate as in Theorem \ref{propcarleman},  with weights that remain bounded as $t \rightarrow T$, i.e., have blow-up only in $t =0$.  For this purpose we introduce the new weights

\begin{equation}\label{pesosnovos}
\begin{array}{llll}
\displaystyle \sigma =e^m, &\displaystyle  \nu(t,x) =  \frac{e^{3\mu \eta(x)}}{\gamma(t)} & \text{and} &\displaystyle m(t,x) = \lambda \frac{e^{3\mu \eta(x)}- e^{5\mu ||\eta||_{\infty}}}{\gamma(t)}, \\ 
\displaystyle \sigma^* =e^{m^*},&   \displaystyle\nu^*(t) = \min_{x\in \overline{\Omega}} \nu(x,t) &   \text{and}  &\displaystyle m^*(t) = \min_{x\in \overline{\Omega}} \nu(x,t),\\
\displaystyle \hat \sigma =e^{\hat m }, & \displaystyle\hat \nu (t) = \max_{x\in \overline{\Omega}} \nu(x,t) & \text{and} &\displaystyle \hat m (t) = \max_{x\in \overline{\Omega}} \nu(x,t), \\
\end{array}
\end{equation}
where $\gamma$ is given by

\begin{equation}
    \gamma(t) = \left\{ \begin{array}{ll}
    t(T-t), & 0 \leq t \leq T/2, \\ 
    T^2/4, & T/2< t\leq T.
    \end{array}\right.
\end{equation}
Combining Carleman estimate \eqref{eqcarleman} with classical energy estimates for the fourth order Schrödinger system, satisfied by $\varphi$ and $\psi$, we can prove the following result.

\begin{proposition}\label{propenergia} With the hypothesis of Proposition \ref{propcarleman} the solution $(\varphi,\psi)$ of \eqref{adjuntoinversa} satisfies the following 
\begin{equation*}
\begin{split}
    ||\varphi||_{L^2(T/2,T;L^2(\Omega))}&+||\varphi_x||_{L^2(T/2,T;L^2(\Omega))} +||\varphi_{xx}||_{L^2(T/2,T;L^2(\Omega))}\\ \leq &\ ||(\varphi,\psi)||_{(L^2(T/4,T/2;L^2(\Omega)))^2} +  ||(g^0,g^1)||_{(L^2(T/2,T;L^2(\Omega)))^2}
    \end{split}
\end{equation*}
and
\begin{equation*}
\begin{split}
    ||\psi||_{L^2(T/2,T;L^2(\Omega))}&+||\psi_x||_{L^2(T/2,T;L^2(\Omega))} +||\psi_{xx}||_{L^2(T/2,T;L^2(\Omega))}\\ \leq &\ ||\psi||_{L^2(T/4,T/2;L^2(\Omega))} +  ||g^1||_{L^2(T/2,T;L^2(\Omega))}.
    \end{split}
\end{equation*}
\end{proposition}

\begin{proof}
Let us consider $\kappa \in C^1([0,T])$ such that 
$$\kappa  = \left\{ 
\begin{array}{l}
0, \quad \text{if} \ t \in [0,T/4], \\
1, \quad \text{if} \ t \in[T/2,T].
\end{array}\right.$$
Note that if $(\varphi,\psi)$ is a solution for \eqref{adjuntoinversa} , so $(\kappa\varphi, \kappa \psi)$ satisfies the following system
\begin{equation}\label{adjuntoregular}
	\left\{\begin{array}{lll}
      i (\kappa\varphi)_t  +  (\kappa\varphi)_{xx} - (\kappa\varphi)_{xxxx} = 1_{\mathcal{O}}(\kappa\psi) +\kappa g^0 +i\kappa_t \varphi, &\text{in}& Q_T,\\
      i (\kappa\psi)_t  +  (\kappa\psi)_{xx}  - (\kappa\psi)_{xxxx} = \kappa g^1 +i\kappa_t \psi,   &\text{in}& Q_T,\\
      (\kappa\varphi)(t,0) = (\kappa\varphi)(t,L) =  (\kappa\varphi)_x(t,0) = (\kappa\varphi)_x(t,L) = 0,  &\text{on}& t \in (0,T), \\
        \psi(t,0) = (\kappa\psi)(t,L) = (\kappa\psi)_x(t,0) = (\kappa \psi)_x(t,L) = 0, &\text{on}& t \in (0,T), \\
      (\kappa\varphi)(T,x)=0, (\kappa\psi)(T,x) =   0, &\text{in}&\Omega.
	\end{array} \right.
	\end{equation}

Now, since $\kappa,\kappa_t \in C([0,T])$ and $C([0,T]) \hookrightarrow L^{\infty}(0,T)$, moreover,  
$\kappa \psi, \kappa_t \psi \in  L^2(0,T;H^2_0(\Omega))$. Then, for $g^1 \in L^2(0,T;H^2_0(\Omega))$ we get that $\kappa \psi$ to satisfy a fourth-order Schr\"odinger system equation with null data and right-hand side in $L^2(0,T;H^2_0(\Omega))$. Therefore, we get that
\begin{equation}\label{eq1energia}
\begin{split}
    \int_{\Omega} |\kappa(t) \psi(t)|dx + \iint_{Q_T} |\kappa \psi_x|dxdt &+ \iint_{Q_T} |\kappa \psi_{xx}|dxdt \\ \leq C &\left(\iint_{Q_T} |\kappa g^1|dxdt + \iint_{Q_T}|\kappa_t \psi|dxdt \right).
        \end{split}
\end{equation}

Multiplying the first equation of \eqref{adjuntoregular} by $\kappa\varphi$ and integrating over $\Omega$ we obtain, after taking the real part and using Young inequality for the integral term of $1_{\mathcal{O}}\kappa \psi \kappa \varphi$, that \begin{equation}\label{eq2energia}
\begin{split}
    -\frac{1}{2}\frac{d}{dt} \int_{\Omega} |\kappa(t)\varphi(t)|^2&+ \int_{\Omega} |\kappa \varphi_x|^2 dx + \int_{\Omega} |\kappa \varphi_{xx}|^2 dx \\ \leq &\ C\left(\int_{\Omega}|\kappa g^0|^2 dx+ \int_{\Omega} |\kappa_t \varphi|^2 dx + \int_{\Omega}|\kappa \psi|^2 dx\right) +\delta \int_{\mathcal{O}} |\kappa \varphi|^2 dx.
    \end{split}
\end{equation}
Finally, integrating \eqref{eq2energia} in $[t,T]$, combining with \eqref{eq1energia} and taking $\delta$ small enough, we get 
\begin{equation}\label{eq3energia}
\begin{split}
   \int_{\Omega} |\kappa(t)\varphi(t)|^2dx+&\iint_{Q_T} |\kappa \varphi_x|^2 dxdt + \iint_{Q_T} |\kappa \varphi_{xx}|^2 dxdt \\ \leq &\ C\left(\iint_{Q_T}|\kappa|^2 (g^0|^2+ |g^1|)^2 dxdt+ \iint_{Q_T} |\kappa_t|^2( \varphi|^2+|\psi|^2) dxdt \right).
    \end{split}
\end{equation}
Therefore, Proposition \ref{propenergia} is a consequence of equations \eqref{eq1energia} and \eqref{eq3energia}.
\end{proof}

As a consequence of the previous result, and due to the definition of \eqref{pesosnovos}, the following Carleman estimate, with new weight functions $\sigma$ and $\nu$, can be obtained. 

\begin{proposition}\label{propcarlemannovo} There exists a constant $C(s,\lambda):=C>0$, such that every solution $(\varphi,\psi)$ of \eqref{adjuntoinversa} satisfies
\begin{equation}\label{eqcarlemanmodificada}
\begin{split}
\iint_{Q_T} &\left( \lambda^7\mu^8\nu^7 \sigma^2 |\varphi|^2+\lambda^5\mu^6\nu^5 \sigma^2 |\varphi_{x}|^2+\lambda^3\mu^4\nu^3 \sigma^2 |\varphi_{xx}|^2 \right) dxdt \\ &+\iint_{Q_T} \left( \lambda^7\mu^8\nu^7 \sigma^2 |\psi|^2+\lambda^5\mu^6\nu^5 \sigma^2 |\psi_{x}|^2+\lambda^3\mu^4\nu^3 \sigma^2 |\psi_{xx}|^2\right) dxdt \\  \leq &\ C\left( \iint_{Q_T} \sigma^2(|g^0|^2+|g^1|^2) dxdt +\lambda  \iint_{\tilde{\omega}_T}\nu \sigma^2|\varphi|^2 dxdt\right) .
    \end{split}
\end{equation}
\end{proposition}
\begin{proof}The result is consequence of Proposition \ref{propenergia}. Indeed, noting that  $\xi=\nu$ and $l=m$, for $t \in [0,T/2]$, and since $l$ is constant in $[T/2,T]$, yields that 
\begin{equation}\label{eq1novopeso}
\begin{split}
\int_{0}^{T/2}\int_{\Omega} &\left( \lambda^7\mu^8\nu^7 \sigma^2 |\varphi|^2+\lambda^5\mu^6\nu^5 \sigma^2 |\varphi_{x}|^2+\lambda^3\mu^4\nu^3 \sigma^2 |\varphi_{xx}|^2 \right) dxdt \\
&+\int_{0}^{T/2}\int_{\Omega} \left( \lambda^7\mu^8\nu^7 \sigma^2 |\psi|^2+\lambda^5\mu^6\nu^5 \sigma^2 |\psi_{x}|^2+\lambda^3\mu^4\nu^3 \sigma^2 |\psi_{xx}|^2\right) dxdt \\ 
=&\int_{0}^{T/2}\int_{\Omega} \left( \lambda^7\mu^8\xi^7 \theta^2 |\varphi|^2+\lambda^5\mu^6\xi^5 \theta^2 |\varphi_{x}|^2+\lambda^3\mu^4\xi^3 \theta^2 |\varphi_{xx}|^2 \right) dxdt \\ &+\int_{0}^{T/2}\int_{\Omega} \left( \lambda^7\mu^8\xi^7 \theta^2 |\psi|^2+\lambda^5\mu^6\xi^5 \theta^2 |\psi_{x}|^2+\lambda^3\mu^4\xi^3 \theta^2 |\psi_{xx}|^2 \right) dxdt.
    \end{split}
\end{equation}
Additionally, for $t\in [T/2,T]$, we have that
\begin{equation}\label{eq2novopeso}
\begin{split}
\int_{T/2}^{T}\int_{\Omega} &\left( \lambda^7\mu^8\xi^7 \theta^2 |\varphi|^2+\lambda^5\mu^6\xi^5 \theta^2 |\varphi_{x}|^2+\lambda^3\mu^4\xi^3 \theta^2 |\varphi_{xx}|^2\right) dxdt \\ &+\int_{T/2}^{T}\int_{\Omega} \left( \lambda^7\mu^8\xi^7 \theta^2 |\psi|^2+\lambda^5\mu^6\xi^5 \theta^2 |\psi_{x}|^2+\lambda^3\mu^4\xi^3 \theta^2 |\psi_{xx}|^2\right) dxdt \\
\leq &\ C\left(\int_{T/2}^{T}\int_{\Omega}(|\varphi|^2+|\varphi_{x}|^2+|\varphi_{xx}|^2+|\psi|^2+|\psi_x|^2+|\psi_{xx}|^2) dxdt \right) \\
\leq &\ C\left( \int_{T/4}^{T/2}\int_{\Omega} (|\varphi|^2 + |\psi|^2 ) dxdt + \int_{T/2}^{T} \int_{\Omega} (|g^0|^2+|g^1|^2)dxdt\right)
\\
\leq &\ C\left( \int_{T/4}^{T/2}\int_{\Omega}\xi^7\theta^2 (|\varphi|^2 + |\psi|^2 ) dxdt + \int_{T/2}^{T} \int_{\Omega} \sigma^2 (|g^0|^2+|g^1|^2)dxdt\right),
    \end{split}
\end{equation}
thanks to Proposition \ref{propenergia}.

Finally,  we note that 
\begin{equation}\label{eq3novopeso}
\begin{split}
 \iint_{Q_T} &\theta^2 (|g^0|^2+|g^1|^2) dxdt + \lambda  \iint_{\tilde{\omega}_T}  \xi \theta^2|\varphi|^2 dxdt \\ \leq &\ C\left( \iint_{Q_T} \sigma^2(|g^0|^2+|g^1|^2) dxdt +\lambda  \iint_{\tilde{\omega}_T}\nu \sigma^2|\varphi|^2 dxdt\right) .
    \end{split}
\end{equation}
Thus, the result follows from \eqref{eqcarleman}, \eqref{eq1novopeso}, \eqref{eq2novopeso} and \eqref{eq3novopeso}. 
\end{proof}

\begin{remark}We point out that Proposition \ref{propenergia} holds by taking the minimum of the weights on the left-hand side and maximum of the weights on the right-hand side of \eqref{eq1energia}.
\end{remark}

\subsection{Null controllability: Linear case}In what follows we use \eqref{eqcarlemanmodificada}, from Proposition \ref{propcarlemannovo}, to deduce the desired null controllability property. Denote  $\mathcal{L}=\mathcal{L}^* = i\partial_t + \partial_{xx} - \partial_{xxxx}$ and introduce the following space
\begin{equation*}
\begin{split}
    \mathcal{C} = &\left\{(u,v,h); (\hat{\sigma})^{-1}u \in L^2(Q_T),(\hat{\sigma})^{-1}v\in L^2(0,T;H^{-2}(\Omega)), (\hat{\nu})^{-\frac{1}{2}}(\hat{\sigma})^{-1}h \in L^2(q_T), \right. \\ &\left. (\nu^*)^{-\frac{7}{2}}(\sigma^{*})^{-1}(\mathcal{L}u - 1_{\omega}h) \in L^2(Q_T), (\nu^*)^{-\frac{7}{2}}(\sigma^{*})^{-1}(\mathcal{L^{*}}v - 1_{\mathcal{O}}u) \in L^2(0,T;H^{-2}(\Omega)), \right. \\& \left. (\hat{\nu})^{-2}\hat{\sigma}^{-1}u \in L^2(0,T;H^4(\Omega))\cap L^{\infty}(0,T;H^2_0(\Omega)),\right.\\&\left. (\hat{\nu})^{-2}\hat{\sigma}^{-1}v \in L^2(0,T;H^2_0(\Omega))\cap L^{\infty}(0,T;L^2(\Omega)),\  v|_{t=T}=0\  \text{in $\Omega$}\right\}.
    \end{split}
\end{equation*}
\begin{remark}
It is important to observe that $\mathcal{C}$ is a Banach space endowed with its natural norm. Additionally, as consequence from the definition of the set $\mathcal{C}$, an element $(u,v,h) \in \mathcal{C}$ is such that $v|_{t=0} = 0$ in $\Omega$. This holds since $(\hat{\nu})^{-2}\hat{\sigma}^{-1}v$ belongs to $ L^{\infty}(0,T;L^2(\Omega))$ and $(\hat{\nu})^{-2}\hat{\sigma}^{-1}$ blow-up only at $t =0$.  
\end{remark}

We are now in a position to prove the null controllability property for solutions of  \eqref{eq:linearized_systeminversa}. The result can be read as follows.

\begin{theorem}\label{controlabilidadelinearizado} Assume the same hypothesis of Proposition \ref{propcarlemannovo}. Additionally, consider 
\begin{equation}\label{eq_control_linear}(\nu^*)^{-\frac{7}{2}}(\hat{\sigma})^{-1} f^0\in L^2(Q_T)\quad \text{and} \quad (\nu^*)^{-\frac{7}{2}}(\hat{\sigma})^{-1} f^1 \in L^2(0,T;H^{-2}(\Omega)).
\end{equation} 
Therefore, we can find a control $h(x,t)=h$ such that the associated solution $(u,v)$ of
\begin{equation}\label{eq:linearized_systeminversa_a}
	\left\{\begin{array}{lll}
      i  u_t  +  u_{xx} - u_{xxxx}  = f^0 + 1_{\omega}h,  &\text{in}& Q_T,\\
      i v_t  + v_{xx} - v_{xxxx} = f^1 +1_{\mathcal{O}}u,   &\text{in}& Q_T,\\
      u(t,0) = u(t,L) = v(t,0) = v(t,L) = 0, &\text{on}& t \in (0,T),\\ 
      u_x(t,0) =  u_x(t,L) = v_x(t,0) = v_x(t,L) = 0,&\text{on}& t \in (0,T), \\
      u(0,x)=u_0(x), v(T,x) = 0, &\text{in}&\Omega,
	\end{array} \right.
	\end{equation}
satisfies $(u,v,h) \in \mathcal{C}$. In particular, $v|_{t=0}\equiv 0$ in $\Omega$.
\end{theorem}

\begin{proof}
We introduce the following spaces 
\begin{equation*}
R_0 =\{u\in H_0^2(\Omega); \ u_{xxxx}\in H_0^2(\Omega) \},
\end{equation*}
\begin{equation*}
Y_0=C([0,T];H^4(\Omega))\cap C^1([0,T];H^2_0(\Omega))
\end{equation*}
and
\begin{equation*}
Y_1=C([0,T];H_0^2(\Omega))\cap C^1([0,T];H^{-2}(\Omega)).
\end{equation*}
Also, let us consider 
\begin{equation*}
\begin{split}
    P_0 =\{(\varphi, \psi)) \in Y_1\times Y_0: \mathcal{L^{*}}v - 1_{\mathcal{O}}u\in L^2(Q_T)\}.
    \end{split}
\end{equation*}
Thanks to Theorem \ref{regular}, $P_0$ is nonempty. Moreover, from now on we will use $\mathcal{L}$ instead of $\mathcal{L^*}$, since both are equal.

Now, define the bilinear form $a:P_0\times P_0\to \mathbb{R}$ by
\begin{equation*}
\begin{split}
    a((\tilde{\varphi},\tilde{\psi}),(\varphi, \psi)):=& \ Re\left(\iint_{Q_T}(\hat \sigma)^2(\mathcal{L}\hat{\varphi} - 1_{\mathcal{O}}\hat{\psi})(\overline{\mathcal{L} \varphi - 1_{\mathcal{O}} \psi}) dxdt+\iint_{Q_T}(\hat \sigma)^2(\mathcal{L}\hat{\psi} )(\overline{\mathcal{L} \psi}) dxdt\right)\\ &+Re\left(\iint_{\tilde{\omega}_T} \hat \nu (\hat \sigma)^2  \hat{\varphi}\overline{\varphi} dxdt \right),
    \end{split}
\end{equation*}
and the linear form $G:P_0\to \mathbb{R}$ given by
\begin{equation*}
    \langle G, (\varphi,\psi) \rangle := Re\iint_{Q_T} f^0 \overline{\varphi}dxdt +\int_0^T\left<f^1,\overline{\psi}\right>dt,
\end{equation*}
where $\left<\cdot,\cdot\right>$ denotes the duality between $H^{-2}(\Omega)$ and $H^2_0(\Omega)$.
Thanks to Proposition \ref{propcarlemannovo}, the bilinear form over $P_0\times P_0$, is sesquilinear, positive, and coercive. Let $P$ be the completion of $P_0$ with the norm induced by $a(\cdot, \cdot)$, in this case, $P$ is a Hilbert space and $a(\cdot, \cdot)$ is well-defined, continuous and coercive bilinear form on $P\times P$. Now, by assumption \eqref{eq_control_linear} and also by Carleman estimate in Proposition \ref{propcarlemannovo}, note that for all $(\varphi, \psi) \in P_0$ we have
\begin{equation*}
\begin{split}
\langle G, (\varphi,\psi) \rangle =& \ Re\iint_{Q_T} f^0 \overline{\varphi}dxdt +\int_0^T\left<f^1,\overline{\psi}\right>dt\\
\leq &\left( \iint_{Q_T} (\nu^*)^7 (\sigma^*)^2(|\varphi|^2 + |\psi|^2) dxdt\right)^{1/2}\\ &\times \left(\iint_{Q_T} (\nu^*)^{-7} (\sigma^*)^{-2}|f^0|^2dxdt + \int_0^T (\nu^*)^{-7} (\sigma^*)^{-2}||f^1||^2_{H^{-2}(\Omega)}dt \right)^{1/2} \\ 
\leq &\ Ca((\varphi,\psi),(\varphi,\psi))^{1/2}\left(\iint_{Q_T} (\nu^*)^{-7} (\sigma^*)^{-2}|f^0|^2dxdt\right. \\&+ \left.\int_0^T (\nu^*)^{-7} (\sigma^*)^{-2}||f^1||^2_{H^{-2}(\Omega)}dt\right)^{1/2},
    \end{split}
\end{equation*}
where we use Young inequality on the first inequality. Hence, $G$ is a bounded functional on $P_0$ and we can extend it continuously to a bounded functional on $P$ due to the Hahn-Banach theorem. Therefore, from the fact that $G$ is a bounded functional on $P_0$ and  $a(\cdot, \cdot)$ is a well-defined, continuous, and coercive bilinear form on $P\times P$, we can use the Lax-Milgram's lemma to conclude that the following variational problem
\begin{equation}\label{eq5linear}
    a((\hat{\varphi},\hat{\psi}),(\varphi, \psi))= \langle G, (\varphi,\psi) \rangle, \quad\forall (\varphi,\psi)\in P,
\end{equation}
has a unique solution $(\hat{\varphi},\hat{ \psi})\in P\times P$.\color{black} 

Let us define $(\hat{u},\hat{v},\hat{h})$ by
\begin{equation}\label{eq6linear}
\left\{\begin{array}{ll}
    \hat{u} = (\hat{\sigma})^2(\mathcal{L}\hat{\varphi} - 1_{\mathcal{O}}\hat{\psi}), & \text{in $Q_T,$} \\
    \hat{v} = (\hat{\sigma})^2\mathcal{L}\hat{\psi}, & \text{in $Q_T,$} \\
    \hat{h} = -\hat\nu(\hat{\sigma})^2  \hat{\varphi} , & \text{in $Q_T,$}
\end{array}\right.
\end{equation}
remembering that $\mathcal{L}^{\star}=\mathcal{L}$. Thanks to \eqref{eq5linear} and \eqref{eq6linear}, we have that
\begin{equation*}
    \iint_{Q_T}(\hat{\sigma})^{-2}\left(|\hat{u}|^2 + |\hat v|^2 +  (\hat{\nu})^{-1}|\hat{h}|^2\right) dxdt = a\left((\hat\varphi, \hat\psi),(\hat\varphi, \hat\psi)\right)< \infty.
\end{equation*}
Considering $(\tilde u, \tilde v)$ be a weak solution of
\begin{equation}\label{eq8linear}
	\left\{\begin{array}{lll}
      i  \tilde u_t  +  \tilde u_{xx} - \tilde u_{xxxx}  = f^0 + 1_{\omega}\tilde h,  &\text{in}& Q_T,\\
      i\tilde  v_t  +\tilde  v_{xx} - \tilde v_{xxxx} = f^1 +1_{\mathcal{O}}\tilde u,   &\text{in}& Q_T,\\
      \tilde u(t,0) = \tilde u(t,L) = \tilde v(t,0) = \tilde v(t,L) = 0,&\text{on}& t \in (0,T),\\  \tilde u_x(t,0) =  \tilde u_x(t,L) = \tilde v_x(t,0) = \tilde v_x(t,L) = 0,&\text{on}& t \in (0,T), \\
\tilde       u(0,x)=0,\tilde  v(T,x) = 0, &\text{in}&\Omega,
	\end{array} \right.
\end{equation}
with control $h= \hat h$ and source terms $f^0$ and $f^1$, since $\tilde h \in L^2(q_T)$, we have, from well-posed result given by Theorem \ref{regular}, that $(\tilde u, \tilde v)$ are well defined.
In the following, we prove that the weak solution $(\hat u , \hat u)$ is a solution by transposition. In fact,  for every $(\varphi, \psi) \in P_0$, it holds from \eqref{eq5linear} and \eqref{eq6linear} that 
\begin{equation}\label{eq9linear}
\begin{split}
&Re\iint_{Q_T}  f^0\overline{ \varphi }dxdt + \int_0^T \left<f^1, \overline{\psi}    \right>_{H^{-2}\times H^2_0}dt + Re\iint_{\tilde{\omega}}  \hat h\overline{\varphi}dxdt \\
&=Re\iint_{Q_T} \hat u(\overline{\mathcal{L} \varphi - 1_{\mathcal{O}} \psi})dxdt + \int_0^T \left<\hat v,\overline{\mathcal{L} \psi}\right>_{H^{-2}\times H^2_0}dt.
    \end{split}
\end{equation}
From \eqref{eq9linear}, we get that
\begin{equation*}
\begin{split}
Re \iint_{Q_T} \hat u\overline{g^0}dxdt &+\int_0^T\left<\hat v,\overline{g^1}\right>_{H^{-2}\times H^2_0}dt = Re\iint_{\tilde{\omega}}  \hat h\overline{\varphi}dxdt\\&+Re\iint_{Q_T}  f^0\overline{ \varphi }dxdt + \int_0^T \left<f^1, \overline{\psi}    \right>_{H^{-2}\times H^2_0}dt,
    \end{split}
\end{equation*}
 for all $(g^0,g^1) \in L^2(0,T;H_0^1(\Omega))$, that is,  $(\hat u , \hat v ) = (\tilde u, \tilde v)$.
 
 Now on,  we prove that solutions $\hat u$ and $\hat v$ of \eqref{eq8linear} are, in fact, more regular. Let us start defining the functions 
$$u_* := (\hat \nu )^{-2}(\hat \sigma)^{-1}\hat u,     \quad \quad v_* :=(\hat \nu )^{-2}(\hat \sigma)^{-1}\hat v,$$
$$f_*^0 := (\hat \nu )^{-2}(\hat \sigma)^{-1}(f^0 + \hat h 1_{\omega}) \quad \text{and}\quad
f_*^1 := (\hat \nu )^{-2}(\hat \sigma)^{-1}f^1.$$
It follows, from \eqref{eq:linearized_systeminversa_a}, that $u_*, v_*, f_*^1$ and $f_*^2$ satisfies the following system
\begin{equation}\label{eq11linear}
\left\{    \begin{array}{ll}
        i(u_*)_t + (u_*)_{xx} - (u_*)_{xxxx} = f_*^0 + i\left( (\hat \nu )^{-2} (\hat \sigma)^{-1}\right)_t \hat u,   & \text{in $Q_T,$} \\[1.5pt]
        i(v_*)_t +(v_*)_{xx} - (v_*)_{xxxx} = f_*^1 + 1_{\mathcal{O}}u_* + i\left( (\hat \nu )^{-2} (\hat \sigma)^{-1}\right)_t \hat v,   & \text{in $Q_T,$} \\
        u_* = v_* = 0, &\text{in $\Sigma$,} \\
        (u_*)_{t=0} = 0 , (v_*)|_{t=T}=0, & \text{in $\Omega$.} 
    \end{array} \right.
\end{equation}
Now, since $ \left(\hat \nu )^{-2} (\hat \sigma)^{-1}\right)_t \leq CT^2 s (\hat \sigma)^{-1}$ we get that $f_*^0 +i\left((\hat \nu)^{-2}(\hat \sigma)^{-1} \right)_t \hat u \in L^2(Q)$ and also \linebreak $f_*^1 +i\left( (\hat \nu )^{-2}(\hat \sigma)^{-1} \right)_t\hat v \in L^2(0,T;H^{-2}(\Omega))$. Now, using the results of Section \ref{Sec1},  for \eqref{eq11linear}, we obtain 
$$u_* \in L^2(0,T;H^4(\Omega))\cap L^{\infty}(0,T; H_0^2(\Omega))$$ and $$v_* \in L^2(0,T;H^2(\Omega))\cap L^{\infty}(0,T; L^2(\Omega)).$$ 
This finishes the proof of Theorem \ref{controlabilidadelinearizado}. 
\end{proof}

\subsection{Null controllability: Nonlinear case}
In this section, we use an inverse mapping theorem to obtain the existence of insensitizing controls for the fourth-order nonlinear Schrödinger equation \eqref{eq:nonlinear_system}. We invite the reader to see the result below as well as additional comments on \cite{Alekseev}.
\begin{theorem}[Inverse mapping theorem]\label{teoinversa} Let $B_1$ and $B_2$ be two Banach spaces and let $$\mathcal{Y}: B_1 \rightarrow B_2$$ satisfying $\mathcal{Y}\in C^1(B_1, B_2)$. Assume that $b_1 \in B_1$, $\mathcal{Y}(b_1) = b_2$ and $$\mathcal{Y}'(b_1):B_1 \rightarrow B_2$$ is surjective. Then, there exists $\delta>0$ such that, for every $b'\in B_2$ satisfying $$||b'- b_2||_{B_{2}}< \delta,$$ there exists a solution of the equation 
$$\mathcal{Y}(b) = b', \quad b \in B_1.$$
\end{theorem}

Finally, we will give the proof of the main result of this manuscript. 

\begin{proof}[Proof of Theorem \ref{thm:control}] Consider, in Theorem \ref{teoinversa}, the following 
$$B_1 = \mathcal{C}\quad \text{and} \quad
B_2 = L^2((\hat \nu)^{-6}(\hat \sigma)^{-3}(0,T);L^2(\Omega))\times L^2((\hat \nu)^{-6}(\hat \sigma)^{-3}(0,T);H^{-2}(\Omega)).$$
Define the operator $$\mathcal{Y}: B_1 \rightarrow B_2$$ such that
$$\mathcal{Y}(u,v,h) :=(\mathcal{L}u -\zeta |u|^2u -1_{\omega}h\,,\, \mathcal{L}v - \overline{\zeta} \overline{u}^2\overline{v} - \overline{\zeta}|u|^2v - 1_{\mathcal{O}}u).$$

\begin{claim}\label{c1}
Operator $\mathcal{Y}$ belongs to $C^1(B_1,B_2)$.
\end{claim}
Indeed, first note that all terms of $\mathcal{Y}$ are linear except: $|u|^2u$, $\overline{u}^2\overline{v}$ and $|u|^2v$. So, the Claim \ref{c1} is equivalent to prove that the trilinear operator given by 
\begin{equation}\label{tril_1}((u_1, v_1, h_1), (u_2, v_2, h_2), (u_3, v_3, h_3)) \mapsto u_1 u_2 v_3
\end{equation}
and
\begin{equation}\label{tril_2_3}((u_1, v_1, h_1), (u_2, v_2, h_2), (u_3, v_3, h_3)) \mapsto u_1 u_2 v_3
\end{equation}
are continuous maps from $\mathcal{C}^3$ to $L^2((\hat \nu)^{-6}(\hat \sigma)^{-1}(0,T);L^2(\Omega))$. However, $(u_i,v_i,h_i) \in \mathcal{C}$, thus we get that $$(\hat{\nu})^{-2}(\hat{\sigma})^{-1}u_i \in L^2(0, T; H^4((\Omega)) \cap L^\infty(0, T; H^2_0(\Omega)) \hookrightarrow L^6(Q_T)
$$
and
$$(\hat{\nu})^{-2}(\hat{\sigma})^{-1}v_i \in L^2(0, T; H^2((\Omega)) \cap L^\infty(0, T; L^2(\Omega)) \hookrightarrow L^6(Q_T),$$
since we are working on an unidimensional case.
At this point, we have fixed $\lambda$ and $\mu$ such that  $$B_2 \subset L^2((\nu^*)^{-\frac{7}{2}}(\sigma^*)
^{-1}(0,T);L^2(\Omega))\times L^2((\nu^*)^{-\frac{7}{2}}(\sigma^*)
^{-1}(0,T);H^{-2}(\Omega))\footnote{
Note that if necessary we could have taken $\lambda$ and $\mu$ large enough in such a way that this inclusion is still satisfied.}$$ holds.
Therefore, note first that
\begin{equation*}
\begin{split}
    \left|\left|(\hat \nu)^{-6}(\hat \sigma)^{-3}u_1 u_2 u_3 \right|\right|_{L^2(Q_T)} &= \left|\left|(\hat \nu)^{-2}(\hat \sigma)^{-1}(\hat \nu)^{-2}(\hat \sigma)^{-1}(\hat \nu)^{-2}(\hat \sigma)^{-1}u_1 u_2 u_3 \right|\right|_{L^2(Q_T)}.
        \end{split}
\end{equation*}
Thus, putting together each $(\hat \nu)^{-2}(\hat \sigma)^{-1}$ with each $u_i$, for $i=1,2,3$, thanks to the previous equality and the Hölder inequality, we get that
\begin{equation}\label{eq1inv}
\begin{split}
\left|\left|(\hat \nu)^{-6}(\hat \sigma)^{-3}u_1 u_2 u_3 \right|\right|_{L^2(Q_T)} \leq C \prod_{k=1}^{3}\left|\left|(\hat \nu)^{-2}(\hat \sigma)^{-1}u_i\right|\right|_{L^6(Q_T)} \leq C \prod_{k=1}^{3} \left|\left|(u_i,v_i,z_i)\right|\right|_{\mathcal{C}},
    \end{split}
\end{equation}
and analogously, we have
\begin{equation*}
    \left|\left|(\hat \nu)^{-6}(\hat \sigma)^{-3}u_1 u_2 v_3 \right|\right|_{L^2(Q_T)} \leq C \prod_{k=1}^{3} \left|\left|(u_i,v_i,z_i)\right|\right|_{\mathcal{C}},
\end{equation*}
this proves that both \eqref{tril_1} and \eqref{tril_2_3} trilinear maps are continuous maps from \linebreak $\mathcal{C}^{3}$ to $L^2((\hat \nu)^{-6}(\hat \sigma )^{-1}(0,T);L^2(\Omega))$, which is equivalent to $\mathcal{Y}$ be a differentiable map from $B_1$ to $B_2$, showing the Claim \ref{c1}.

\begin{claim}\label{c2}
$\mathcal{Y}'(0,0,0)$ is surjective.
\end{claim}

First, note that $\mathcal{Y}(0,0,0)=(0,0)$. By the other hand, observe that $\mathcal{Y}'(0,0,0):B_1 \rightarrow B_2$ is given by 
$$\mathcal{Y}'(0,0,0)(u,v,h) = (iu_t + u_{xx} - u_{xxxx} - 1_{\omega}h,  iv_t + v_{xx} - v_{xxxx} - 1_{\mathcal{O}u}),$$
for $(u,v,h) \in B_1$. Invoking the null controllability result for linear system \eqref{eq:linearized_systeminversa}, that is, thanks to Theorem \ref{controlabilidadelinearizado}, $\mathcal{Y}'(0,0,0)$ is surjective, proving the Claim \ref{c2}. 

Finally, by taking $b_1 = (0,0,0)$, $b_2 = (0,0)$ and using Theorem \ref{teoinversa}, there exists $\delta>0$ such that if $||(f^0,f^1)||_{B_2} < \delta$, then we can find a control $h$ such that the triple $(u,v,h) \in B_1$ satisfies $\mathcal{Y}(u,v,h) = (f^0, f^1)$. By a particular choice of $f
^0 = f \in L^2((\hat \nu)^{-6}(\hat \sigma)^{-1}(0,T);L^2(\Omega))$ and $f^1 \equiv 0$, Theorem \ref{thm:control} is showed since a triple $(u,v,h) \in B_1$ satisfies $v(0) = 0$ in $\Omega$ and solves \eqref{eq:nonlinear_system}. 
\end{proof}


\section{Further comments and open issues}\label{SecFinal}
To our knowledge, these results in this article are the first concerning the existence of insensitizing controls for the fourth-order Schrödinger equation, in this way, we believe that this manuscript can open a series of questions, which are discussed now.

\subsection{Null condition of the initial data}
In this point, we discuss the necessity to assume the null condition of the initial data in Theorem \ref{thm:control}. In \cite{teresa2000insensitizing}, the author proves that under some suitable conditions, the existence of insensitizing controls may or may not hold, which indicates that this kind of problem cannot be solved for every initial data. In this way, we also have the same drawback in our result. To overcome this difficulty, we believe that the techniques used for the Heat equation, due to De Tereza \cite{teresa2000insensitizing}, can be adapted for our case. Precisely, the idea consists of using the fundamental solution to construct an explicit solution where the observability inequality does not hold.

\subsection{About the nonlinear terms}
Note that if we change the cubic term $|u|^2u$ by a more general term $|u|^{p-2}u$, with $p\geq3$, then one must prove a partial null controllability for the following system
\begin{equation*}
	\left\{\begin{array}{lll}
      i u_t  + u_{xx} - u_{xxxx} - \zeta |u|^{p-2} u = f + 1_{\omega}h,  &\text{in}& Q,\\
      i  v_t  + v_{xx} - v_{xxxx} - \overline{\zeta} p|u|^{p}\overline{u}^{2}\overline{v} - (p+1)\overline{\zeta} |u|^{p-2}\overline{v} = 1_{\mathcal{O}}u,  &\text{in}& Q,\\
      u(t,0) = u(t,L) = v(t,0) = v(t,L) = 0,&\text{on}& t \in (0,T),\\ u_x(t,0) = u_x(t,L) = v_x(t,0) = v_x(t,L) = 0,&\text{on}& t \in (0,T), \\
         u(0,x)=u_0(x), v(T,x) = 0, &\text{in}&\Omega.
	\end{array} \right.
	\end{equation*}
If the structure of the problem is still the same and we only change the nonlinearity, the main difficulty here is to obtain well-posedness results which gives enough regularity for the solutions to obtain the analogous Hölder estimate as in \eqref{eq1inv}. In fact, to solve  it one must have valid embedding from the state spaces into $L^{2p-2}(Q)$, for $p\geq3$, which is possible since the following estimate holds
$$    \left|\left||u|^{p-2} u\right|\right|_{H^s(\Omega)}\leq C||u||^{p-1} _{H^s(\Omega)},$$
when $s\geq\frac{1}{2}$ and $p\geq3$, see \cite{LiZheng2020} for well-posedness of the general nonlinear problem.
 
Additionally, if we change to a general type nonlinearity $g$, we obtain the following optimal system 
\begin{equation*}
	\left\{\begin{array}{lll}
      i u_t  + u_{xx} - u_{xxxx} +g(u) = f + 1_{\omega}h,  &\text{in}& Q,\\
      i  v_t  + v_{xx} - v_{xxxx} +g'(u)v = 1_{\mathcal{O}}u,  &\text{in}& Q,\\
       u(t,0) = u(t,L) = v(t,0) = v(t,L) = 0,&\text{on}& t \in (0,T),\\ u_x(t,0) = u_x(t,L) = v_x(t,0) = v_x(t,L) = 0,&\text{on}& t \in (0,T), \\
         u(0,x)=u_0(x), v(T,x) = 0, &\text{in}&\Omega.
	\end{array} \right.
	\end{equation*}
It is expected that most of these problems have no solution, i.e., it is not possible to insensitize the functional unless we impose some conditions on $g$. To exemplify the comments above, some of these issues were already considered for the case of nonlinearities with superlinear growth at infinity. In \cite{bodart2004existence}, the authors dealing with a semilinear heat equation proved positive result of existence of insensitizing controls considering $g\in C^1$ a nonlinear function verifying $g''\in L^{\infty}_{loc}(\mathbb{R})$, $g(0) = 0$ and  $$\lim_{|s|\rightarrow \infty} \frac{g'(s)}{\ln(1+|s|)} = 0,$$ furthermore, the result is also valid for nonlinearities $g$ of the form  $$|g(s)| = |p_1(s)|\ln^{\alpha}(1+|p_2(s)|),$$ for all $|s|\geq s_0 >0,$ with $\alpha \in [0,1)$ and $p_i$, $i=1,2$, are affine functions. Moreover, they proved negative results of existence considering a nonlinearity $g$ verifying the conditions above, that is,  taking $g$ as $$g(s) = \int_{0}^{|s|}\ln^{\alpha}(1+|\sigma|^2) d\sigma, \quad \text{for all $s\in \mathbb{R}$},$$ but choosing $\alpha >2$. Similar results are proved in \cite{zl} for a class of nonlinear Ginzburg-Landau equation.

Thus, in the case of the fourth-order nonlinear Schrödinger equation, this kind of situation, that is,  introducing a function $g$ with certain properties and proving the existence of insensitizing controls is still an open issue.

\subsection{About the sentinel functional}\label{sb1} One way to solve the problems of nonlinearity is to change the structure of the functional. Due to the lack of regularity of the characteristic function, if we change it to a more regular function then one can still prove the result for more general nonlinearity $|u|^{p-2}u$, with $p\geq3$, considering a functional of the form 
$$J(\tau, h) = \frac{1}{2}\iint_{Q_T} \mathcal{R}(x)|u(x,t)|^2 dxdt,$$
where $\mathcal{R}\in C^{\infty}(\Omega)$ is a smooth function with $supp(\mathcal{R}) \subset \mathcal{O}$. 

We note that there exists uncountable insensitizing control problems as we change the sentinel functional. In fact, by the equivalent formulation in a cascade system with double the equations of the original system, controllability problems with fewer control forces than equations are not fully understood in PDEs, so they can also be interesting from the control theory point of view.  Some of the motivations for these problems arise from physical phenomena, thus typically we focus our attention on functionals that have ``physical" meanings: If the functional is the local $L^2-$norm of the solution then we are looking for controls that locally preserve the energy (kinetic or potential, depending on the modeling) of the system, and if we change to a first derivative (or gradient in the $N-$dimensional case) the problem consists in finding controls that locally preserves the mean value of the energy. 

In this perspective, let $D$ be a derivative operator such as $Du = u_x$ or $Du= u_{xx}$. An interesting $-$ and difficult $-$ problem is to analyze the existence of insensitizing controls when the sentinel functional takes the form  $$J(\tau, h) :=\frac{1}{2} \iint_{Q_{T}}|Du(x,t)|^2 dxdt.$$ In such, the optimal system become
\begin{equation*}
	\left\{\begin{array}{lll}
      i u_t  + u_{xx} - u_{xxxx} - \zeta |u|^2 u = f + 1_{\omega}h,  &\text{in}& Q_T,\\
      i  v_t  + v_{xx} - v_{xxxx} - \overline{\zeta} \overline{u}^2\overline{v} - 2\overline{\zeta} |u|^2 v = D(1_{\mathcal{O}}Du),   &\text{in}& Q_T,\\
      u(t,0) = u(t,L) = v(t,0) = v(t,L) = 0,&\text{on}& t \in (0,T), \\
      u_x(t,0) = u_x(t,L) = v_x(t,0) = v_x(t,L) = 0,&\text{on}& t \in (0,T), \\
         u(0,x)=u_0(x), v(T,x) = 0, &\text{in}&\Omega.
	\end{array} \right.
	\end{equation*}
Again, it is not expected to obtain positive results of the existence of insensitizing controls for every differential operator $D$ in virtue of the lack of regularity provoked by coupling term $D(1_{\mathcal{O}}Du)$, since (again) the characteristic function is not regular. Despite that, Guerrero \cite{guerrero2007null} dealt with a parabolic equation. The author proved a positive result of existence considering a functional depending on the gradient of the solution. Since the equation was linear, with constant coefficients, the argument consisted of considering a global Carleman estimate with different exponents, not for the equation, but for the equation satisfied by the Laplacian of the solutions to then recover information using the equation with the coupling. This is not the case when dealing with a nonlinear problem since deriving the equation would give us many other terms. In \cite{guerrero2007controllability}, the same author proved a similar result considering a linear Stokes equation with constant coefficients but with for the $curl$ of the solution. Finally, we cite the work of the second author \cite{tanaka2019}, where the authors proved positive results of insensitizing controls considering a functional depending on the gradient of the solution for the cubic nonlinear Ginzburg-Landau equation. The result arose by proving a new suitable Carleman estimate for the Ginzburg-Landau equation. 

In this spirit, there are many alternatives to define the sentinel functional related to the insensitizing control problems for 4NLS. Thus, we expect that these three works together with the results in this paper, open prospects to prove similar results considering a sentinel functional with the gradient of the solution. Moreover, since the Carleman estimate \eqref{eqcarleman} has third-order terms, maybe it is possible, at some point, to adapt the arguments to consider a functional with the Laplacian of the solution of the 4NLS, but clearly, to prove it is necessary new arguments of those that were applied here, at least proving a new Carleman estimate for the fourth-order Schrödinger equation, as was done in \cite[Theorem 1.1]{peng2} for the Cahn-Hilliard type equation and as in \cite{peng3} where a Carleman estimate for stochastic fourth order Schr\"odinger equation is showed. The readers are invited to read the recent and interesting work by Imanuvilov and Yamamoto \cite{ImYa2020}, which proves a Carleman estimate for a fourth-order parabolic equation in general dimensions.

\subsection{N--dimensional case} Zheng and Zhou \cite{zheng} studied the boundary controllability of the 4NLS in a bounded domain $\Omega\subset\mathbb{R}^n$. Using a $L^2-$Neumann boundary control, the authors proved that the solution of 4NLS is exactly controllable in $H^{-2}(\Omega)$ using the Hilbert Uniqueness Method and multiplier techniques.  In the sense of the existence of insensitizing controls, we conjecture that the Carleman inequality shown here can be extended to the $N$-dimensional case. Thus, if we consider the sentinel functional as defined in \eqref{def:jfunc}, our result remains valid, for this case. However, the main issue here is when we consider a functional like the one mentioned in Subsection \ref{sb1} or other types of functional associated with the nonlinear problem. This type of problem looks interesting and still is open for the fourth-order nonlinear Schrödinger equation.

\subsection*{Acknowledgments} The authors are grateful to the anonymous referees for the constructive comments that improved this work.

\appendix

\section{Well-posedness}\label{Sec1}

In this section, we will show some results about the existence of a solution for the system 
\begin{equation}\label{sistemalinearexistencia}
	\left\{\begin{array}{lll}
      i  u_t  +  u_{xx} - u_{xxxx}  = F^0, &\text{in}& {Q_T},\\
      i v_t  + v_{xx} - v_{xxxx} = F^1 +1_{\mathcal{O}}u,   &\text{in}& {Q_T},\\
      u(t,0) = u(t,L) = u_x(t,0) = u_x(t,L) = 0,&\text{on}& t \in (0,T),\\  v(t,0) =  v(t,L) = v_x(t,0) = v_x(t,L) = 0,&\text{on}& t \in (0,T), \\
      u(0,x)=u_0(x), v(T,x) = 0, &\text{in}&\Omega,
	\end{array} \right.
	\end{equation}
for given $u_0$ and $(F^0, F^1)$. The proofs here can be adapted to prove the existence of solutions for systems \eqref{eq:linearized_systeminversa} and \eqref{adjuntoinversa}.

\subsection{The linearized system}
We first consider the simplest linear equation with null boundary conditions which is a linearized version of \eqref{biharmonic} around zero. More precisely, we consider the following
\begin{equation}\label{eqlinearexistence}
    \left\{ \begin{array}{ll}
        iu_t + u_{xx} - u_{xxxx} = f, & \text{in $Q_T$,} \\
        u(t,0) = u(t,L) =u_x(t,0) = u_x(t,L) =0, & \text{on $(0,T)$,} \\
        u(0,x) = u_0(x), & \text{in $\Omega$.}
    \end{array}\right.
\end{equation}
The first result is a consequence of the semigroup theory.  Before presenting it, let us consider the differential operator $A: D(A)\subset L^2(\Omega) \rightarrow L^2(\Omega)$ given by $$Au:= iu_{xx}-iu_{xxxx},$$ with domain $D(A)=H^4(\Omega)\cap H_0^2(\Omega)$. Thus, the nonhomogeneous linear system \eqref{eqlinearexistence} takes the form
\begin{equation}\label{eqlinearexistenceabstrato}
\begin{cases}
u_t (t)= Au(t)+if(t),  & t\in[0,T], \\
u(0)=u_0.
\end{cases}
\end{equation}
The following proposition guarantees some properties for the operator $A$. Precisely, the result ensures the existence of regular solutions for the system \eqref{eqlinearexistence}.
\begin{proposition}\label{pro_linear} Let $f \in C^1([0,T];L^2(\Omega))$ and $u_0 \in D(A)$, then \eqref{eqlinearexistence} has a unique solution 
\begin{equation}\label{regularidadeeqlinearizado}
    u \in C([0, T];H^4(\Omega)\cap H^2_0(\Omega)) \cap C^1([0, T ]; L^2(\Omega)).
\end{equation} 
\end{proposition}

\begin{proof}
Consider the linear operator defined by $A$.  This allows us to rewrite \eqref{eqlinearexistence} in the abstract form \eqref{eqlinearexistenceabstrato}.  We have that $A$ is skew-adjoint operator and $A$ is m-dissipative. Indeed, first, is not difficult to see that
$$(Au,v)_{L^2(\Omega)}=-(u,Av)_{L^2(\Omega)},$$
for all $u,v\in D(A).$ That is, $A$ is symmetric. Additionally, $D(A^{\star})=D(A)$, so $A$ is skew-adjoint. Finally, we have
\begin{equation*}
   ( Au , u )_{L^2(\Omega)} = Re\left(i\int_{\Omega} (u_{xx} - u_{xxxx})\overline{u} dx\right) = Re\left(i \int_{\Omega} -(|u_x|^2 + |u_{xx}|^2)dx \right) = 0,
\end{equation*}
for any $u\in D(A)$, and then $A$ is dissipative. Therefore, $A$ is an m-dissipative operator (e.g. \cite[Corollary 2.4.8]{Ca}) and by the Hille–Yosida–Phillips theorem (e.g. \cite[Theorem 3.4.4]{Ca}) we obtain that $A$ is a generator of a contraction semigroup in $L^2(\Omega)$. Thus, if $u_0\in D(A)$ and $f\in C^1([0,T];L^2(\Omega))$, then equation \eqref{eqlinearexistence} has solutions $u$ with the regularity \eqref{regularidadeeqlinearizado}.  (e.g. \cite[Proposition 4.1.6]{Ca}).
\end{proof}

\subsection{The coupled linearized system}

We are now concerned with the existence of solutions for the coupled linearized system. More precisely, we will prove the well-posedness results to the system \eqref{sistemalinearexistencia}.  First, consider the linear unbounded operator $A$ defined in the previous subsection and 
\begin{equation*}
    \left\{ \begin{array}{ll}
A_1u = -iu_{xx} + iu_{xxxx}\in H^{-2}(\Omega),\\
D(A_1) =H_0^2(\Omega).
    \end{array}\right.
\end{equation*}
Both operators are m-dissipative with dense domains; therefore, they generate the $C_0$
semigroups of contractions $\mathcal{S}_0$ and $\mathcal{S}_1$, respectively. Now, consider the following spaces:
\begin{equation*}
Y_0=C([0,T];D(A))\cap C^1([0,T];L^2(\Omega))
\end{equation*}
and
\begin{equation*}
Y_1=C([0,T];D(A_1))\cap C^1([0,T];H^{-2}(\Omega)).
\end{equation*}
The next result is dedicated to proving the existence of \textit{regular solutions} for \eqref{sistemalinearexistencia}
\begin{theorem}[Regular solutions]\label{regular}
Assume that $u_{0} \in D\left(A\right)$,
$$F^{0} \in C\left([0, T], H^4(\Omega)\cap H^2_0(\Omega)\right) \cap W^{1,1}\left(0, T ; H^4(\Omega)\cap H^2_0(\Omega)\right)$$
and
$$F^{1} \in C\left([0, T], H^{-2}(\Omega)\right) \cap W^{1,1}\left(0, T ; H^{-2}(\Omega)\right).$$
Then, problem \eqref{sistemalinearexistencia} has a unique regular solution in the sense that
\[
\left\{\begin{array}{l}
(u, v) \in Y_{0} \times Y_{1}, \\
i  u_t  +  u_{xx} - u_{xxxx}  = F^0,\\
i v_t  + v_{xx} - v_{xxxx} = F^1 +1_{\mathcal{O}}u,  \\
\left.u\right|_{t=0}=u_{0},\left.v\right|_{t=T}=0.
\end{array}\right.
\]
\end{theorem}
\begin{proof}
Note that, thanks to \cite[Proposition 4.1.6]{Ca}, we get that the mild solution 
$$
u(t)=\mathcal{S}_{0}(t) u_{0}+\int_{0}^{t} \mathcal{S}_{0}(t-s) F^{0}(s) d s\in Y_{0}
$$
verifies
\begin{equation}\label{ut_1}
\begin{cases}
i  u_t  +  u_{xx} - u_{xxxx}  = F^0,\\
\left.u\right|_{t=0}=u_{0}.
\end{cases}
\end{equation}
Now, it is not difficult to see that
\[
\begin{aligned}
\int_{0}^{T}\left\|\left(u 1_{\mathcal{O}}\right)(s)\right\|_{H^{-2}(\Omega)}^{2} ds &\leq\sup_{\zeta \in H_{0}^{2}(\Omega),	 \|\zeta\|_{H_{0}^{2}(\Omega)}=1}\int_{0}^{T} \int_{\mathcal{O}} |u| \cdot \varsigma dx  d t  \leq \iint_{Q_{T}}\left|u\right|^{2} dx dt,
\end{aligned}
\]
and hence $u 1_{\mathcal{O}} \in C\left([0, T], H^{-2}(\Omega)\right) \cap W^{1,1}\left(0, T ; H^{-2}(\Omega)\right) .$ Then, applying again \cite[Proposition 4.1.6]{Ca}, and we get that the mild solution 
$$v(t)=\int_{t}^{T} \mathcal{S}_{1}(s-t)\left(F^{1}+1_{\mathcal{O}}u\right)(s) ds \in Y_{1}$$
 satisfies
\begin{equation}\label{vt_1}
\begin{cases}
i v_t  + v_{xx} - v_{xxxx} = F^1 +1_{\mathcal{O}}u,  \\
\left.v\right|_{t=T}=0.
\end{cases}
\end{equation}
Thus, Theorem \ref{regular} is achieved putting together $(u,v)$ satisfying \eqref{ut_1} and \eqref{vt_1}.
\end{proof}

\subsection{Transposition solutions} In what follows, we will talk about transposition solutions that are of particular interest for the purposes of this paper.  

\begin{definition}
Let $u_0\in L^2(\Omega)$ and $(F^0,F^1)\in[L^2(0,T;H^{-2}(\Omega))]^2$. We say that a pair $$(u,v)\in L^2(0,T;H^2_0(\Omega))\times L^2(Q_T)$$ is a solution in the transposition sense of \eqref{sistemalinearexistencia}, if it satisfies
\begin{equation}\label{transposition}
\begin{split}
&\int_{0}^{T}\left\langle g^{0}, u\right\rangle_{H^{-2} H_{0}^{2}}dt=\operatorname{Re} \int_{\Omega} \varphi(0) \overline{u_{0}} dx+\int_{0}^{T}\left\langle F^{0}, \varphi\right\rangle_{H^{-2} H_{0}^{2}}  dt \\
&\operatorname{Re} \iint_{Q_{T}} g^{1} \bar{v} dx dt=\operatorname{Re} \int_{0}^{T} \int_{\mathcal{O}}  u \cdot \overline{\psi} dx dt+\int_{0}^{T}\left\langle F^{1}, \psi\right\rangle_{H^{-2} H_{0}^{2}} dt
\end{split}
\end{equation}
for every $\left(g^{0}, g^{1}\right) \in L^{2}\left(0, T ; H^{-2}(\Omega)\right) \times L^{2}\left(Q_{T}\right)$, where $\left\langle \cdot,\cdot\right\rangle$ denotes the duality between $H^{-2}(\Omega$ and $H_{0}^{2}(\Omega)$,  and $(\varphi, \psi)$ is the solution of
\begin{equation}\label{adjuntoinversa_a}
	\left\{\begin{array}{lll}
      i \varphi_t  +  \varphi_{xx} - \varphi_{xxxx} = g^0,  &\text{in}& Q_T,\\
      i \psi_t  +  \psi_{xx}  - \psi_{xxxx} = g^1,   &\text{in}& Q_T,\\
      \varphi(t,0) = \varphi(t,L) =  \varphi_x(t,0) = \varphi_x(t,L) = 0, &\text{on}& t \in (0,T),\\  \psi(t,0) = \psi(t,L) = \psi_x(t,0) = \psi_x(t,L) = 0, &\text{on}& t \in (0,T), \\
      \varphi(T,x)=0, \psi(0,x) = 0, &\text{in}&\Omega.
	\end{array} \right.
	\end{equation}
	\end{definition}
	We have the following result about the existence and uniqueness of transposition
solutions.
\begin{theorem}For $u_{0} \in L^{2}(\Omega)$ and $\left(F^{0}, F^{1}\right) \in\left[L^{2}(0, T;H^{-2}(\Omega))\right]^{2}$, there exists a unique $(u, v) \in\left[L^{2}\left(0, T ; H_{0}^{2}(\Omega)\right)\right]^{2}$ satisfying \eqref{transposition} for every $\left(g^{0}, g^{1}\right) \in L^{2}\left(0, T ; H^{-2}(\Omega)\right) \times L^{2}\left(Q_{T}\right)$, where $(\varphi, \psi)$ is solution of \eqref{adjuntoinversa_a}.
\end{theorem}
\begin{proof}Let $\Psi_1:L^2(0,T;H^2_0(\Omega))\to\mathbb{R}$ the operator defined by
$$\Psi_1(h^0)=\operatorname{Re}\int_{\Omega} \varphi(0) \overline{u_{0}} dx+\int_{0}^{T}\left\langle F^{0}, \varphi\right\rangle_{H^{-2} H_{0}^{2}}  dt ,$$ where $\varphi$ satisfies the first equation of \eqref{adjuntoinversa_a} for $g^0:=h^0_{xxxx}\in L^2(0,T;H^{-2}(\Omega))$. From the energy estimates, it is easy to see the continuity of $\Psi_1$. Then, thanks to the Lax-Milgram theorem, there exists $u\in L^2(0,T;H^2_0(\Omega))$ such that 
$$\int_{0}^{T}\left\langle g^{0}, u\right\rangle_{H^{-2} H_{0}^{2}}dt=\operatorname{Re}\iint_{Q_T}h^0_{xx}\overline{u}_{xx}dxdt=\Psi_1(h^0),$$
for every $g^0\in H^{-2}(\Omega)$, with $g^0=h^0_{xxxx}$.  Analogously, we have the existence of $v\in L^2(Q_T)$ satisfying the second equation of \eqref{transposition}, since the linear form 
$$\Psi_1(g^1)=\operatorname{Re} \int_{0}^{T} \int_{\mathcal{O}}  u \cdot \overline{\psi} dx dt+\int_{0}^{T}\left\langle F^{1}, \psi\right\rangle_{H^{-2} H_{0}^{2}} dt$$
is continuous in $L^2(Q_T)$.

\vspace{0.2cm}
\noindent\textbf{Claim:} We have that $v$ belongs to $L^2(0,T;H^2_0(\Omega))$.
\vspace{0.1cm}

Indeed,  first, we take sequences of regular data such that $u_{0}^{n} \rightarrow u_{0}$ in $L^{2}(\Omega)$ and $\left(F_{n}^{0}, F_{n}^{1}\right) \rightarrow\left(F^{0}, F^{1}\right)$ in $L^{2}\left(0, T ; H^{-2}(\Omega)\right) \times L^{2}\left(Q_{T}\right)$. We show that the regular solutions $\left(u_{n}, v_{n}\right)$ for \eqref{sistemalinearexistencia} (whose existence is given in Theorem \ref{regular}) with initial data $u_{0}^{n}$ and $\left(F_{n}^{0}, F_{n}^{1}\right)$ on the right-hand side, are also a solution in the transposition sense; moreover, it is bounded in $\left[L^{2}\left(0, T ; H_{0}^{2}(\Omega)\right)\right]^{2}$. Hence, in the limit, we obtain that $(u, v) \in\left[L^{2}\left(0, T ; H_{0}^{2}(\Omega)\right)\right]^{2}$.

\vspace{0.1cm}
Finally,  for uniqueness,  suppose that $(\hat{u},\hat{v})$ is another solution of \eqref{sistemalinearexistencia}.  Thus, 
$$\operatorname{Re} \int_{0}^{T}\left\langle g^{0}, u-\hat{u}\right\rangle_{H^{-2} H_{0}^{2}}dt=0\quad \text{and} \quad \operatorname{Re} \iint_{Q T} g^{1}(v-\hat{v}) dx dt=\operatorname{Re} \int_{0}^{T} \int_{\mathcal{O}}(u-\hat{u})\cdot \overline{\psi}dxdt, $$
for all $g^0\in L^2(0,T;H^{-2}(\Omega))$ and $g^{1} \in L^{2}\left(Q_{T}\right).$
Hence, $y=\hat{y}$ and $z=\hat{z}$.
\end{proof}


\end{document}